\documentclass{amsart}
\usepackage[margin=2cm]{geometry}
\usepackage{amssymb}
\usepackage{enumerate}   
\usepackage{graphicx}
\usepackage{mathtools}
\usepackage{tikz-cd} 
\usepackage{color}
\usepackage{setspace}
\usepackage[colorlinks]{hyperref}
\graphicspath{{figures/}}

\newtheorem{theorem}{Theorem}[section]
\newtheorem{corollary}[theorem]{Corollary}

\theoremstyle{definition}
\newtheorem{definition}[theorem]{Definition}
\newtheorem{conjecture}[theorem]{Conjecture}

\newtheorem{proposition}[theorem]{Proposition}
\newtheorem{example}[theorem]{Example}
\newtheorem{remark}[theorem]{Remark}
\newtheorem{lemma}[theorem]{Lemma}

\def\Zbb{\mathbb{Z}}
\def\Rbb{\mathbb{R}}
\def\Cbb{\mathbb{C}}
\def\Nbb{\mathbb{N}}

\def\Tcal{\mathcal{T}}

\def\Ecal{\mathcal{E}}

\def\Xcal{\mathcal{X}}

\def\CS{\mathrm{CS}}
\def\tr{\mathrm{tr}}	
\def\sl{\mathfrak{sl}_2(\Cbb)}

\newcommand{\bigzero}{\mbox{\normalfont\Large\bfseries 0}}
\newcommand{\rvline}{\hspace*{-\arraycolsep}\vline\hspace*{-\arraycolsep}}

\begin{document}
	
	\title[ ]{The twisted 1-loop invariant\\and the Jacobian of Ptolemy coordinates}
	
	%    Information for first author

	\author[ ]{Seokbeom Yoon}
	\address{Departament de Matem\`atiques, Universitat Aut\`onoma de Barcelona, 08193 Cerdanyola del Vall\`es, Spain}
	\email{sbyoon15@mat.uab.cat}

	%    General info
%	 \subjclass[2020]{57K31, 57K32}
	
	 \date{\today}
	
	% \dedicatory{This paper is dedicated to our advisors.}
	
	
	\begin{abstract} 
				
		We present an alternative definition of the twisted 1-loop invariant in terms of the Jacobian of Ptolemy coordinates.
		As an application, we prove that the twisted 1-loop invariant is equal to the adjoint twisted Alexander polynomial for all hyperbolic once-punctured torus bundles.
		This implies that the 1-loop conjecture proposed by Dimofte and Garoufalidis holds for all hyperbolic once-punctured torus bundles.
		
	\end{abstract} 
	
	\maketitle
%	\tableofcontents

	\section{Introduction} \label{sec.intro}

	An ideal triangulation of a cusped hyperbolic 3-manifold  is an effective tool to describe the hyperbolic structure of the manifold: each tetrahedron is assigned to a complex variable which describes its shape and each edge has an equation to  guarantee that the tetrahdra around the edge are well-glued compatibly~\cite{Thurston}.
	These equations can be encoded into integer matrices called the gluing equation matrices.
	Surprisingly, (a slight modification of) the gluing equation matrices turned out to have the symplectic property \cite{NZ}, and  this led to some new invariants of 3-manifolds defined in terms of ideal triangulations.
	The \emph{1-loop invariant} introduced in~\cite{DG1} is one of such invariants which is also motivated by the generalized Volume conjecture~\cite{Guk05}.
	Although it is defined in terms of ideal triangulations, it has topological invariance, namely, it is invariant under 2--3 Pachner moves.
	Furthermore, the 1-loop conjecture~ \cite[Conj.1.8]{DG1} predicts that the 1-loop invariant coincides with the adjoint Reidemeister torsion~\cite{Porti:torsion}. 
	To state the conjecture precisely, let $M$ be a 1-cusped hyperbolic 3-manifold and $\Tcal$ be an ideal triangulation of $M$.
	We denote by $\tau_\gamma(M)$ the adjoint Reidemeister torsion of $M$ and by $\tau_\gamma^\CS(\Tcal)$ the 1-loop invariant of $\Tcal$ (here $\CS$ stands for the Chern--Simons theory, as the 1-loop invariant is derived from its perturbative partition function).  Note that both are elements of the trace field of $M$ well-defined up to sign and depend on a choice of a peripheral curve $\gamma$ of $M$.
	
%		
	
%	It also depends  on a choice of a peripheral curve $\gamma$, as the twisted cohomology $H^\ast(M;\sl)$ is non-trivial.
	
	\begin{conjecture}[\cite{DG1}] \label{conj.1loop}
		The 1-loop invariant $\tau^\CS_\gamma(\Tcal)$ of $\Tcal$ is equal to the adjoint Reidemeister torsion $\tau_\gamma(M)$ of $M$ for all peripheral curves $\gamma$.
	\end{conjecture}
%	Here the adjoint Reidemeister torsion $\tau_\gamma(M)$ is a combinatorial torsion of $M$ with the Lie algebra coefficient  $\sl$  twisted by the adjoint action $\mathrm{Ad}\rho$ associated to the geometric representation $\rho:\pi_1(M)\rightarrow \mathrm{PSL}_2(\Cbb)$ where the basis of  the twisted cohomology $H^1(M;\sl)$ is determined by the peripheral curve $\gamma$. We refer to  ~\cite{Porti:torsion} for details.
	
	 Garoufalidis and the author showed  in  \cite{GY21} that the gluing equation matrices can be generalized to  \emph{twisted gluing equations matrices} (whose entries are in $\Zbb[t^{\pm1}]$) by 
	 choosing an infinite cyclic cover of $M$, or an epimorphism $\alpha : \pi_1(M) \twoheadrightarrow \Zbb$  equivalently, and additionally recording how the gluing data are lifted to a fundamental domain. They also defined \emph{twisted 1-loop invariant} using a similar construction to the 1-loop invariant and conjectured that it  is equal to the  adjoint twisted Alexander polynomial (associated with the geometric representation and the epimorphism $\alpha$). Note that both are Laurent polynomials in one variable $t$ over the trace field and well-defined up to multiplication of  units in $\Zbb[t^{\pm1}]$.
		
	\begin{conjecture}[\cite{GY21}]  \label{conj.twist1loop}
		The twisted 1-loop invariant $\tau^\CS(\Tcal,t)$ of $\Tcal$ is equal to the adjoint twisted Alexander  polynomial $\tau(M,t)$ of $M$.
	\end{conjecture}

	Our original motivation was to study Conjecture~\ref{conj.twist1loop} for punctured surface bundles over $S^1$.
	As is well-known, twisted Alexander polynomials of those manifolds are simply determined by monodromies.
	In particular,
	adjoint twisted Alexander polynomials are determined by how monodromies act on the $\mathrm{SL}_2(\Cbb)$-character varieties of punctured surfaces.
	This leads us to consider so-called Ptolemy coordinates or assignments, instead of shape parameters of tetrahedra, as we need coordinates for the character varieties of punctured surfaces.
	Ptolemy coordinates 
	are well-studied objects in both 2 and 3-dimensional hyperbolic geometry but in different languages (see, for instance, \cite{fock2006moduli,penner2012decorated,GTZ}).
	An essential property of Ptolemy coordinates is that they effectively parameterize (certain subsets of)  character varieties of both punctured surfaces and cusped 3-manifolds in the same fashion.
	This allows us to compute adjoint twisted Alexander polynomials of punctured surface bundles easily in terms of Ptolemy coordinates (see Theorem~\ref{thm.ATAP} for once-punctured torus bundles).
	Note that we will use a sign-extended version of Ptolemy assignments,   $\sigma$-Ptolemy assignments, developed in \cite{Yoon19} (there, these were called $\sigma$-deformed Ptolemy assignments) due to a technical reason that the geometric representation of a cusped hyperbolic manifold does not admit an $\mathrm{SL}_2(\Cbb)$-lift which is boundary-unipotent \cite{Cal06}.

	The goal of this paper is to relate  the twisted 1-loop invariant with Ptolemy coordinates. Precisely, we present an alternative definition of the twisted 1-loop invariant as well as the 1-loop invariant in terms of the Jacobian of Ptolemy coordinates~~(see Theorem~\ref{thm.twistPtolemy} and Corollary~\ref{cor.1loop}).
	An advantage of this definition is that it does not require a combinatorial flattening, an extra datum which is needed in the original definition of (twisted) 1-loop invariant  in \cite{DG1} and \cite{GY21}.
	Interestingly, this resembles the fact that we need a combinatorial flattening to compute the Chern-Simons invariant in terms of the shapes of tetrahedra \cite{neumann2004extended} but we do not if we use Ptolemy coordinates \cite{Zic09, GTZ}.

%	We note that we will use a sign-extended version of Ptolemy coordinates,   $\sigma$-Ptolemy coordinates developed in \cite{Yoon19} (there, these were called $\sigma$-deformed Ptolemy coordinates), since the geometric representation of cusped hyperbolic 3-manifold does not admit an $\mathrm{SL}_2(\Cbb)$-lift which is boundary-unipotent.

	As an application of Theorem~\ref{thm.twistPtolemy}, 
	we prove that Conjecture~\ref{conj.twist1loop} is true for all hyperbolic once-punctured torus bundles (which was our original motivation; most of the techniques used in the proof can be applied to higher genus surface bundles). To state the theorem precisely, let $M_\varphi$ be the once-punctured torus bundle over $S^1$ with monodromy  $\varphi$ and   $\alpha :\pi_1(M_\varphi) \twoheadrightarrow \Zbb$ be the epimophism induced from the projection map $M_\varphi \rightarrow S^1$. 
	We choose an ideal triangulation $\Tcal_\varphi$ of $M_\varphi$ as the canonical triangulation (also called the monodromy triangulation)~\cite{FH82,lackenby2003canonical,Gue06}.
	
	\begin{theorem} \label{thm.main2}
		The twisted 1-loop invariant  $\tau^\CS(\Tcal_\varphi,t)$  equals to the adjoint twisted Alexander polynomial $\tau(M_\varphi,t)$ of $M_\varphi$ for all hyperbolic once-punctured torus bundles $M_\varphi$.
	\end{theorem}

%	
%	Ptolemy assignments play a role of bridge connecting the twist 1-loop invariant and the adjoint twisted Alexander polynomial for once-punctured torus bundles.
%	
	
	As a corollary of Theorem~\ref{thm.main2}, we obtain that the 1-loop conjecture~\cite[Conj.1.8]{DG1}  holds for all hyperbolic once-punctured torus bundles.
	
	\begin{corollary} \label{thm.main3}
		The 1-loop invariant $\tau^\CS_\gamma(\Tcal_\varphi)$  equals to the adjoint Reidemeister torsion $\tau_\gamma(M_\varphi)$  of $M_\varphi$ for  all peripheral curves $\gamma$ and hyperbolic once-punctured torus bundles $M_\varphi$.
	\end{corollary}

	The paper is organized as follows. In Section~\ref{sec.twist}, we briefly recall the definition of the twisted 1-loop invariant.
	In Section~\ref{sec.jac}, we present an alternative definition of the twisted 1-loop invariant in terms of the Jacobian of Ptolemy coordinates.
	In Section~\ref{sec.POTB}, we explicitly compute the twisted 1-loop invariant (Section~\ref{sec.twistOPT}) and Ptolemy assignments (Section~\ref{sec.PtolemyOncePun})  	for once-punctured torus bundles
	and present a formula for the adjoint twisted Alexander polynomial in terms of Ptolemy assignments (Section~\ref{sec.ATAP}).
	Based on computations in Section~\ref{sec.POTB}, we prove Theorem~\ref{thm.main2} and Corollary~\ref{thm.main3} in Section~\ref{sec.proof}. \\
	
	\noindent \textbf{Acknowledgments.} 
	The author would like to thank Jinsung Park who gave motivation of this work in the first place, and Stavros Garoufalidis, Hyuk Kim, Seonhwa Kim and Joan Porti for their helpful comments. 
	The author was supported by Basic Science Research Program through the NRF of Korea funded by the Ministry of Education (2020R1A6A3A03037901).

	\section{The twisted 1-loop invariant}  \label{sec.twist}
	Let $M$ be a 1-cusped hyperbolic 3-manifold equipped with an epimorphism $\alpha : \pi_1(M) \twoheadrightarrow \Zbb$
	and $\Tcal$ be an ideal triangulation of $M$.
	Let $\widetilde{M}$ be the infinite cyclic cover of $M$ corresponding to the kernel of $\alpha$ and  $\widetilde{\Tcal}$ be the ideal triangulation of $\widetilde{M}$ induced from $\Tcal$.
	We index the edges $e_i$ and the tetrahedra 
	$\Delta_j$ of $\Tcal$ by $1 \leq i,j \leq N$ (the number of tetrahedra is equal to that of edges, as the Euler characteristic of $M$ is zero).
	We choose a fundamental domain of $M$ in $\widetilde{M}$ and denote by $\widetilde{e}_i$ and $\widetilde{\Delta}_j$ the lifts of $e_i$ and $\Delta_j$, respectively, to the fundamental domain.	 Then an edge and a tetrahedron of $\widetilde{\Tcal}$ are given by $g^k \cdot \widetilde{e}_i$ and $g^k \cdot \widetilde{\Delta}_j$ for some $k \in \Zbb$ and $1 \leq i,j \leq N$ where $g$ is a generator of the deck transformation group $\Zbb$ of $\widetilde{M}$.

	A quad type
	of a tetrahedron is a pair of opposite edges; hence each tetrahedron has three different
	quad types. We choose a quad type and an orientation of each $\Delta_j$ so that each edge of
	$\Delta_j$ admits a shape parameter among 
	\[z_j, \quad z'_j=\frac{1}{1-z_j}, \quad z''_j=1- \frac{1}{z_j} \in \Cbb \setminus \{0,1\}\]
	with opposite edges having same parameters (see Figure \ref{fig.tetrahedron}).
	We denote by $\square$ the three pairs of opposite edges of a tetrahedron
	so that the edges of $\square$ are assigned to the shape parameter $z^\square$.
	
	\begin{figure}[htpb!]
		%% Creator: Inkscape 1.0.2 (e86c8708, 2021-01-15), www.inkscape.org
%% PDF/EPS/PS + LaTeX output extension by Johan Engelen, 2010
%% Accompanies image file 'tetrahedron.pdf' (pdf, eps, ps)
%%
%% To include the image in your LaTeX document, write
%%   \input{<filename>.pdf_tex}
%%  instead of
%%   \includegraphics{<filename>.pdf}
%% To scale the image, write
%%   \def\svgwidth{<desired width>}
%%   \input{<filename>.pdf_tex}
%%  instead of
%%   \includegraphics[width=<desired width>]{<filename>.pdf}
%%
%% Images with a different path to the parent latex file can
%% be accessed with the `import' package (which may need to be
%% installed) using
%%   \usepackage{import}
%% in the preamble, and then including the image with
%%   \import{<path to file>}{<filename>.pdf_tex}
%% Alternatively, one can specify
%%   \graphicspath{{<path to file>/}}
%% 
%% For more information, please see info/svg-inkscape on CTAN:
%%   http://tug.ctan.org/tex-archive/info/svg-inkscape
%%
\begingroup%
  \makeatletter%
  \providecommand\color[2][]{%
    \errmessage{(Inkscape) Color is used for the text in Inkscape, but the package 'color.sty' is not loaded}%
    \renewcommand\color[2][]{}%
  }%
  \providecommand\transparent[1]{%
    \errmessage{(Inkscape) Transparency is used (non-zero) for the text in Inkscape, but the package 'transparent.sty' is not loaded}%
    \renewcommand\transparent[1]{}%
  }%
  \providecommand\rotatebox[2]{#2}%
  \newcommand*\fsize{\dimexpr\f@size pt\relax}%
  \newcommand*\lineheight[1]{\fontsize{\fsize}{#1\fsize}\selectfont}%
  \ifx\svgwidth\undefined%
    \setlength{\unitlength}{114.47727954bp}%
    \ifx\svgscale\undefined%
      \relax%
    \else%
      \setlength{\unitlength}{\unitlength * \real{\svgscale}}%
    \fi%
  \else%
    \setlength{\unitlength}{\svgwidth}%
  \fi%
  \global\let\svgwidth\undefined%
  \global\let\svgscale\undefined%
  \makeatother%
  \begin{picture}(1,0.84370128)%
    \lineheight{1}%
    \setlength\tabcolsep{0pt}%
    \put(0,0){\includegraphics[width=\unitlength,page=1]{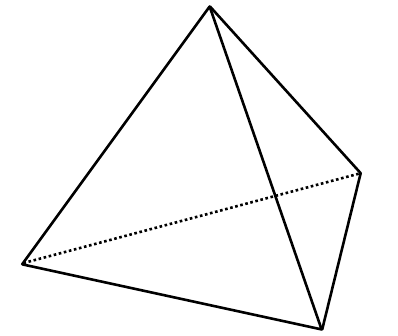}}%
    \put(0.22612007,0.52498993){\makebox(0,0)[lt]{\lineheight{1.25}\smash{\begin{tabular}[t]{l}$z$\end{tabular}}}}%
    \put(0.88209508,0.18215681){\makebox(0,0)[lt]{\lineheight{1.25}\smash{\begin{tabular}[t]{l}$z$\end{tabular}}}}%
    \put(0.72426566,0.63173901){\makebox(0,0)[lt]{\lineheight{1.25}\smash{\begin{tabular}[t]{l}$z''$\end{tabular}}}}%
    \put(0,0){\includegraphics[width=\unitlength,page=2]{tetrahedron.pdf}}%
    \put(0.60780402,0.46139955){\makebox(0,0)[lt]{\lineheight{1.25}\smash{\begin{tabular}[t]{l}$z'$\end{tabular}}}}%
    \put(0,0){\includegraphics[width=\unitlength,page=3]{tetrahedron.pdf}}%
    \put(0.45336339,0.2695276){\makebox(0,0)[lt]{\lineheight{1.25}\smash{\begin{tabular}[t]{l}$z'$\end{tabular}}}}%
    \put(0,0){\includegraphics[width=\unitlength,page=4]{tetrahedron.pdf}}%
    \put(0.42299074,0.07344528){\color[rgb]{0,0,0}\makebox(0,0)[lt]{\lineheight{1.25}\smash{\begin{tabular}[t]{l}$z''$\end{tabular}}}}%
  \end{picture}%
\endgroup%

		\caption{An ideal tetrahedron with shape parameters $z^\square$.}
		\label{fig.tetrahedron}
	\end{figure}
	
	The gluing equation matrices $\mathbf{G}^\square$ of $\Tcal$ are the $N \times N$  matrices (whose rows and columns are indexed by the edges and by the tetrahedra of $\Tcal$
	respectively and) whose $(i,j)$-entry $\mathbf{G}^\square_{ij}$ is the number
	of edges of $\Delta_j$ with parameter $z^\square_j$ that are incident to $e_i$ in $\Tcal$. 
	The gluing equation matrices basically show which tetrahedra are attached to  $e_i$. In particular, the gluing equation of $e_i$  is given by
	\begin{equation} \label{eqn.ge}
		e_i : \quad	\prod_{j=1}^N z_j^{\mathbf{G}_{ij}} {z_j'}^{\mathbf{G}_{ij}'}{z_j''}^{\mathbf{G}_{ij}''} = 1 .
	\end{equation}
	In a similar manner, we consider tetrahedra of $\widetilde{\Tcal}$ that are attached to $\widetilde{e}_i$. Precisely, for $k \in \Zbb$ let $\mathbf{G}_k^\square$ be the
	$N \times N$  matrix
	whose $(i,j)$-entry is the number of edges of $g^k \cdot \widetilde{\Delta}_j$ with
	parameter $z_j^\square$ that are incident to  $\widetilde{e}_i$ in $\widetilde{\Tcal}$. We then define the \emph{twisted gluing equation matrices} $\mathbf{G}^\square(t)$ of $\Tcal$  by
	\begin{equation*}
	\mathbf{G}^\square(t):=\sum_{k \in \Zbb} \mathbf{G}^\square_k\, t^k .
	\end{equation*}
	Note that entries of $\mathbf{G}^\square(t)$ are in $\Zbb[t^{\pm1}]$ and that $\mathbf{G}^\square(t)$  reduces to $\mathbf{G}^\square$ when $t=1$.

		A peripheral curve is an oriented, homotopically non-trivial,
	simple closed curve in the peripheral torus of $M$. Each peripheral curve
	$\gamma$ is associated with three row vectors $\mathbf{C}_\gamma^\square\in\Zbb^N$ that describe its completeness equation as
	\begin{equation} \label{eqn.comp}
		\gamma : \quad	\prod_{j=1}^N z_j^{\mathbf{C}_{\gamma j}} {z_j'}^{\mathbf{C}_{\gamma j}'}{z_j''}^{\mathbf{C}_{\gamma j}''} = 1
		\end{equation}
	where $\mathbf{C}^\square_{\gamma j}$ is the $j$-th component of $\mathbf{C}^\square_\gamma$. 	 We say that a point $(z_1,\ldots, z_N) \in (\Cbb \setminus \{0,1\})^N$ is a \emph{solution} of $\Tcal$ if it satisfies the gluing equation~\eqref{eqn.ge} for all edges of $\Tcal$ and the completeness equation~\eqref{eqn.comp} for all peripheral curves of $M$ and that a solution is \emph{geometric} if its holonomy representation $\pi_1(M)\rightarrow \mathrm{PSL}_2(\Cbb)$ is the geometric representation of $M$.
	Throughout the paper, we assume that a geometric solution of $\Tcal$ exists.
%	 namely, $\Tcal$ supports the geometric representation of $M$.
	
 	A \emph{combinatorial flattening} of $\Tcal$ consists of three column vectors $f^\square \in \Zbb^N$ satisfying
	\begin{align}
		f+ f' +  f'' & =(1,\ldots,1)^T, \label{eqn.f1}\\
		\mathbf{G} f+\mathbf{G}' f' + \mathbf{G}'' f'' & =(2,\ldots,2)^T, \label{eqn.f2}\\
		\mathbf{C}_\gamma f+\mathbf{C}'_\gamma f' + \mathbf{C}''_\gamma f'' &=0	 \label{eqn.f3}
	\end{align}
	for all peripheral curves $\gamma$. 
	Note that every ideal triangulation has a combinatorial flattening~\cite[Thm.4.5]{Neumann04}.

	\begin{definition}[\cite{GY21}]
	With the above notation, the \emph{twisted 1-loop invariant} $\tau^\CS(\Tcal,t)$ of $\Tcal$ is defined as
	\begin{equation} \label{eqn.twist1loop}
		\tau^\CS(\Tcal,t)
		:= \frac{\det\big(
			\mathbf{G}(t)\, \mathrm{diag}(\zeta)+\mathbf{G}'(t) \,
			\mathrm{diag}(\zeta')+\mathbf{G}''(t) \,
			\mathrm{diag}(\zeta'') \big)}{
			\displaystyle\prod_{j=1}^N\zeta_j^{f_j} {\zeta_j'}^{f_j'} { \zeta_j''}^{f_j''}} \in \Cbb[t^{\pm1}]
	\end{equation}
	where the right-hand side is evaluated at the geometric solution of $\Tcal$, 
	\begin{equation} \label{eqn.zeta}
		\zeta_j = \frac{d \log z_j}{dz_j}=\frac{1}{z_j}, \quad
		\zeta_j' = \frac{d \log z_j'}{dz_j}=\frac{1}{1-z_j},\quad
		\zeta_j'' = \frac{d \log z_j''}{dz_j}=\frac{1}{z_j(z_j-1)},
	\end{equation}
	and $\mathrm{diag}(\zeta^\square)$ is the diagonal matrix with diagonal entries $\zeta^\square_1,\ldots, \zeta^\square_n$.		
	\end{definition}	
	Several choices, such as quad types of $\Delta_j$, a fundamental domain of $M$ and a flattening of $\Tcal$, are involved in the right-hand side of Equation~\eqref{eqn.twist1loop}. However, it turns out that
	the twisted 1-loop invariant $\tau^\CS(\Tcal,t)$ is well-defined up to multiplication of  units in $\Zbb[t^{\pm1}]$. Furthermore, it is invariant under 2--3 Pachner moves.
	We refer to \cite{GY21} for details.
	
	\begin{remark} \label{rmk.partial}
		The logarithmic form of the gluing equation \eqref{eqn.ge} is 
		\begin{equation} \label{eqn.gelog}
		 e_i:\quad \sum_{j=1}^N  \mathbf{G}_{ij} \log z_j  +  {\mathbf{G}_{ij}'} \log {z_j'} + {\mathbf{G}_{ij}''} \log {z_j''} = 2 \pi \sqrt{-1} \, .
		\end{equation}
		Since we fixed a fundamental domain, say $D$, of $M$, the infinite cyclic cover $\widetilde{M}$ is divided  into $g^k \cdot D$ for $k \in \Zbb$. Accordingly, the left-hand side of Equation \eqref{eqn.gelog} is divided with respect to $k\in \Zbb$: the contribution from $g^k \cdot D$ is
	\begin{equation} \nonumber 
		 \sum_{j=1}^N  (\mathbf{G}_k)_{ij} \log z_j  +  {(\mathbf{G}'_k)_{ij}} \log {z_j'} + {(\mathbf{G}''_k)_{ij}} \log {z_j''}	 \, .	 
 	 \end{equation}
	It naturally allows us to associate a Laurent polynomial $E_i $ to each edge $e_i$	
	\begin{equation}  \nonumber
		E_i=\sum_{k \in \Zbb} \left(\sum_{j=1}^N  (\mathbf{G}_k)_{ij} \log z_j  +  {(\mathbf{G}'_k)_{ij}} \log {z_j'} + {(\mathbf{G}''_k)_{ij}} \log {z_j''}	\right) t^k 
	\end{equation}
	which reduces to the left-hand side of Equation \eqref{eqn.gelog} when $t=1$. It follows from the definition~\eqref{eqn.zeta} that
	\begin{equation} \nonumber
		\big(
		\mathbf{G}(t)\, \mathrm{diag}(\zeta)+\mathbf{G}'(t) \,
		\mathrm{diag}(\zeta')+\mathbf{G}''(t) \,
		\mathrm{diag}(\zeta'') \big)_{ij} = \frac{\partial E_i}{\partial z_j}
	\end{equation}
	for $1 \leq i,j \leq N$.
	\end{remark}

\section{The Jacobian of $\sigma$-Ptolemy coordinates} \label{sec.jac}

In this section, we present an alternative definition of the twisted 1-loop invariant in terms of the Jacobian of $\sigma$-Ptolemy coordinates (see Theorem~\ref{thm.twistPtolemy} below).

\subsection{A single tetrahedron} \label{sec.tet}
We first consider a single tetrahedron $\Delta$ with its \emph{truncation} $\overline{\Delta}$, a polyhedron obtained from $\Delta$ by chopping off its vertices.	
There are two types of edges in $\overline{\Delta}$: an edge of a triangular face is called a \emph{short-edge} and an edge that is not a short-edge is called a \emph{long-edge}. 
Let $\delta \overline{\Delta}$ be the union of four triangular faces of $\overline{\Delta}$ so that the edges of $\delta \overline{\Delta}$ are the short-edges of $\overline{\Delta}$.
We denote by $\Ecal(X)$ for $X\in\{\Delta,\overline{\Delta}, \delta \overline{\Delta}\}$ the set of oriented edges of $X$. 
It is clear that $\Ecal( \delta \overline{\Delta}) \subset \Ecal(\overline{\Delta})$ and that $\Ecal(\Delta)$ is identified with $\Ecal(\overline{\Delta}) \setminus \Ecal(\delta \overline{\Delta})$ in an obvious way.

We say that an assignment $\sigma : \Ecal(\delta \overline{\Delta}) \rightarrow \{\pm1\}$ of a sign to each oriented short-edge is a \emph{cocycle} if
\begin{itemize}
	\item $\sigma(-e)= \sigma(e)^{-1}$ for $e \in \Ecal(\delta \overline{\Delta})$  where  $-e$ is the same edge $e$ but with the reversed orientation;
	\item 
	$\sigma(e_1) \sigma(e_2) \sigma (e_3) =1$ if $e_1,e_2$ and  $e_3$ form the boundary of a 2-cell of $\delta \overline{\Delta}$ in any cyclic order. 	
\end{itemize}
In fact, the orientations of short-edges do not matter as the target group is $\{\pm1\}$.

\begin{definition}[\cite{Yoon19}]
	For a cocycle $\sigma : \Ecal(\delta \overline{\Delta}) \rightarrow \{\pm1\}$
	an assignment  $c : \Ecal(\Delta) \rightarrow \Cbb^\ast$ is   called a \emph{$\sigma$-Ptolemy assignment (or coordinates)} if $c(-e)=-c(e)$ for all $e \in \Ecal(\Delta)$ and  
	\begin{equation} \label{eqn.sgPtolemy}
	c(l_{02})c(l_{13}) - \frac{\sigma(s_{031})\sigma(s_{302})}{\sigma(s_{312})\sigma(s_{021})} c(l_{03})c(l_{12}) - \frac{\sigma(s_{013})\sigma(s_{102})}{\sigma(s_{132})\sigma(s_{023})} c(l_{01})c(l_{23})   = 0
	\end{equation}
	where  $l_{ij}$ is the oriented edge of $\Delta=[v_0,v_1,v_2,v_3]$ starting from the vertex $v_i$ and heading to $v_j$ and $s_{ijk}$ is the oriented short-edge starting from $l_{ij}$ and heading to $l_{jk}$ as in Figure~\ref{fig.truncation}.  We refer to Equation~\eqref{eqn.sgPtolemy} as the \emph{$\sigma$-Ptolemy equation} $p^\sigma_\Delta$ of $\Delta$.		
	\begin{figure}[htpb!]
		\centering
		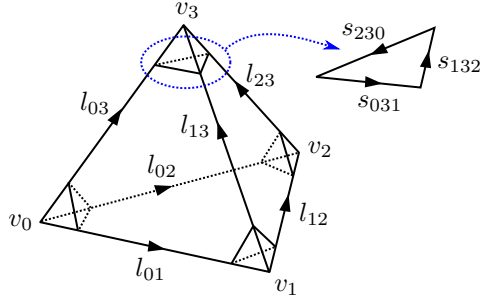
		\caption{An edge notation for a truncated tetrahedron.}
		\label{fig.truncation}
	\end{figure}	
\end{definition}
It is shown in \cite{Yoon19} that a  $\sigma$-Ptolemy assignment $c : \Ecal(\Delta) \rightarrow \Cbb^\ast$  determines  a cocycle $\rho : \Ecal(\overline{\Delta}) \rightarrow \mathrm{SL}_2(\Cbb)$ and the shape of $\Delta$.
Precisely, the cocycle $\rho$ is given by
\begin{equation} \label{eqn.Passign}
	\rho(l_{ij}) = \begin{pmatrix} 0 & -c(l_{ij})^{-1} \\
		c(l_{ij}) & 0 \end{pmatrix}, \quad
	\rho(s_{ijk}) = \begin{pmatrix}  \sigma(s_{ijk}) & -\frac{\sigma(s_{kij})}{\sigma(s_{jki})} \frac{ c(l_{ki})}{c(l_{ij})\,c(l_{jk})} \\ 0 & \sigma(s_{ijk})^{-1} \end{pmatrix} 
\end{equation}
for Figure~\ref{fig.truncation}
and the shape parameter $z$ at the edge $l_{13}$ is given by		
\begin{equation} \label{eqn.Zshape}
	z=  -\frac{\sigma(s_{103})\sigma(s_{321})}{\sigma(s_{032})\sigma(s_{210})}\,\frac{c(l_{01}) c(l_{23})}{c(l_{03})c(l_{12})} 	\, .
\end{equation}	
We refer to \cite[Sec.2]{Yoon19} for details.

\subsubsection{The Jacobian of the $\sigma$-Ptolemy equation}
We choose a quad type of $\Delta$ and an orientation of each edge of $\Delta$.
We denote by $e^\square_0$ and $e^\square_1$ two edges of $\Delta$ with shape parameter $z^\square$.

\begin{lemma} 	\label{lem.key}
	With the above notation, we have
	\[	\frac{c(e^\square_i)}{c(e_0)c(e_1)}\frac{\partial p^\sigma_\Delta}{\partial c(e^\square_i)}=\pm \frac{\zeta^{\square}}{\zeta }\]
	for $i=0,1$ and any $\square$ where the sign $\pm$ does not depend on $i$ and $\square$.
\end{lemma}
\begin{proof} Suppose that the edges of $\Delta$ are oriented as in Figure~\ref{fig.truncation} and  a quad type of $\Delta$ is given by $(e_0,e_1)=(l_{02}, l_{13})$ so that $(e'_0, e'_1)= (l_{01},l_{23})$ and $(e''_0,e''_1)=(l_{03},l_{12})$.
	It follows that for $i=0,1$
	\begin{align} 
		\allowdisplaybreaks
		\frac{c(e_i)}{c(e_0)c(e_1)}\frac{\partial p^\sigma_\Delta}{\partial c(e_i)}&=1 = \frac{\zeta}{\zeta}, \nonumber \\[5pt]
		\frac{c(e'_i)}{c(e_0)c(e_1)}\frac{\partial p^\sigma_\Delta}{\partial c(e'_i)}&= - \frac{ \sigma(s_{013})\sigma(s_{102}) }{\sigma(s_{132})\sigma(s_{023})} \frac{c(l_{01}) c(l_{23})}{c(l_{02})c(l_{13})} \nonumber  \\ 
		&=-  
		\frac{c(l_{01}) c(l_{23})}{\frac{\sigma(s_{132})\sigma(s_{023})}{ \sigma(s_{013})\sigma(s_{102}) } \frac{\sigma(s_{031})\sigma(s_{302})}{\sigma(s_{312})\sigma(s_{021})} c(l_{03})c(l_{12}) + c(l_{01})c(l_{23})}\nonumber \\ \label{eqn.simplifying1}
		&=\frac{1}{1/z -1 }  = \frac{\zeta'}{\zeta},  \\[5pt]
		\frac{c(e''_i)}{c(e_0)c(e_1)}\frac{\partial p^\sigma_\Delta}{\partial c(e''_i)}&= - \frac{ \sigma(s_{031})\sigma(s_{302}) }{\sigma(s_{312})\sigma(s_{021})} \frac{c(l_{03}) c(l_{12})}{c(l_{02})c(l_{13})} \nonumber \\
		&=
		-  \frac{c(l_{03}) c(l_{12})}{ c(l_{03})c(l_{12}) + \frac{\sigma(s_{312})\sigma(s_{021})}{ \sigma(s_{031})\sigma(s_{302}) } \frac{\sigma(s_{013})\sigma(s_{102})}{\sigma(s_{132})\sigma(s_{023})} c(l_{01})c(l_{23})} \nonumber \\
		&=\frac{1}{-1+z }  = \frac{\zeta''}{\zeta}.		\label{eqn.simplifying2}
	\end{align}
	Note that Equations~\eqref{eqn.simplifying1}~and~\eqref{eqn.simplifying2} are followed from Equation~\eqref{eqn.Zshape} with the fact that $\sigma$ is a cocycle. If we change orientations of some edges of $\Delta$, then there could be an additional sign in the above computation, but it would not depend on $i$ and $\square$. We compute the other quad types $(l_{01},l_{23})$ and $(l_{03},l_{12})$ of $\Delta$ similarly.
\end{proof}

\begin{lemma} \label{lem.key2}
	Let $(f,f',f'') \in \Zbb^3$ be a triple of integers satisfying $f+f'+f''=1$. Then we have
	\begin{equation} \label{eqn.key2}
		\frac{c(e^\square_i)}{c^f_\Delta}\frac{\partial p^\sigma_\Delta}{\partial c(e^\square_i)}=\pm \frac{\zeta^\square}{\zeta^f_\Delta}
	\end{equation}
	for $i=0,1$ and any $\square$ where the sign $\pm$ does not depend on $i$ and $\square$, and
	\begin{align*}
		c^f_\Delta &:= c(e_0)^f c(e_1)^fc(e'_0)^{f'} c(e'_1)^{f'}c(e''_0)^{f''}c(e''_1)^{f''},\\
		\zeta^f_\Delta &:=		\zeta^f \zeta'^{f'} \zeta''^{f''}.
	\end{align*}
\end{lemma}
\begin{proof}  As in the proof of Lemma~\ref{lem.key}, suppose that a quad type of $\Delta$ is  $(e_0,e_1)=(l_{02},l_{13})$ so that $(e'_0, e'_1)= (l_{01},l_{23})$ and $(e''_0,e''_1)=(l_{03},l_{12})$.  Then we have
	\begin{align*}
		\zeta^f \zeta'^{f'} \zeta''^{f''} & = \zeta^{1-f'-f''}\zeta'^{f'} \zeta''^{f''} \\
		&= \zeta \left(\frac{\zeta'}{\zeta}\right)^{f'}\left(\frac{\zeta''}{\zeta}\right)^{f''}\\
		&=\pm \zeta \left(\frac{c(l_{01})c(l_{23})}{c(l_{02})c(l_{13})}\right)^{f'}\left(\frac{c(l_{03})c(l_{12})}{c(l_{02})c(l_{13})}\right)^{f''} \\
		&= \pm \zeta \cdot \frac{ c(e_0)^f c(e_1)^fc(e'_0)^{f'} c(e'_1)^{f'}c(e''_0)^{f''}c(e''_1)^{f''} }{c(e_0)c(e_1)}.
	\end{align*}
	where the sign $\pm$ above only depends on the cocycle $\sigma$ and the triple $(f,f',f'')$.
	Combining this with Lemma~\ref{lem.key}, we obtain Equation~\eqref{eqn.key2}.
\end{proof}

Note that Lemma~\ref{lem.key2} reduces to Lemma~\ref{lem.key} when $(f,f',f'')=(1,0,0)$.

\subsection{$\sigma$-Ptolemy assignments for a 3-manifold}  \label{sec.ptolemy}
The notion of $\sigma$-Ptolemy assignment in Section~\ref{sec.tet} naturally extends to an ideal triangulation of a 3-manifold \cite{Yoon19}. We briefly recall its definition with some properties that we need in the following sections.

Let $M$ be a 1-cusped hyperbolic 3-manifold and  $\overline{M}$ be the compact manifold with a torus boundary whose interior is homeomorphic to $M$.
An ideal triangulation $\Tcal$ of $M$  endows $\overline{M}$ with a decomposition into truncated tetrahedra whose triangular faces triangulate the boundary $\partial \overline{M}$ of $\overline{M}$.
For simplicity we confuse this decomposition of $\overline{M}$ as well as that of $\partial \overline{M}$ with its underlying space.
We denote by $\Ecal(X)$ for $X\in\{\Tcal,\overline{M}, \partial \overline{M}\}$ the set of oriented edges of $X$. 

\begin{definition}[\cite{Yoon19}]
	For a cocycle	$\sigma : \Ecal(\partial \overline{M})  \rightarrow \{\pm1\}$ an assignment $c : \Ecal(\Tcal) \rightarrow \Cbb^\ast$ is called a \emph{$\sigma$-Ptolemy assignment (or coordinates)} if it satisfies  $c(-e)=-c(e)$ for all $e \in \Ecal(\Tcal)$ and the $\sigma$-Ptolemy equation
	for all tetrahedra of $\Tcal$.	
\end{definition}

As in Equation~\eqref{eqn.Passign},  a $\sigma$-Ptolemy assignment $c : \Ecal(\Tcal) \rightarrow \Cbb^\ast$ determines a cocycle $\rho : \Ecal(\overline{M}) \rightarrow \mathrm{SL}_2(\Cbb)$ and thus induces a representation $\rho: \pi_1(M)\rightarrow \mathrm{SL}_2(\Cbb)$ (also denoted by $\rho$) up to conjugation. 
We obtain a map $\Phi$ from the set $P^\sigma(\Tcal)$ of all $\sigma$-Ptolemy assignments of $\Tcal$ to the $\mathrm{SL}_2(\Cbb)$-character variety $\Xcal(M)$ of $M$
\begin{equation} \label{eqn.Phi}
	\Phi :P^\sigma(\Tcal)\rightarrow \Xcal(M)
\end{equation}
by sending $c \in P^\sigma(\Tcal)$ to the conjugacy class of $\rho$.
Note that by the definition~\eqref{eqn.Passign} the representation $\rho$ satisfies (up to conjugation)
\begin{equation} \label{eqn.Compatible}
	\rho(\gamma) = \begin{pmatrix} \sigma(\gamma) & \ast \\ 0 & \sigma(\gamma)^{-1}  \end{pmatrix} \quad \textrm{for all } \gamma \in \pi_1(\partial \overline{M})
\end{equation}   where $\sigma : \pi_1(\partial \overline{M})  \rightarrow \{ \pm1 \}$ is the homomorphism (also denoted by $\sigma$)  induced from the cocycle $\sigma$.

\begin{remark}	\label{rmk.scaling}
	As the $\sigma$-Ptolemy equation~\eqref{eqn.sgPtolemy} is homogeneous, the set $P^\sigma(\Tcal)$ admits the multiplication by a non-zero complex number acting diagonally.
	It is proved in \cite[Prop.2.6]{Yoon19} that $\Phi(k \cdot c) = \Phi(c)$ for all $ c\in P^\sigma(\Tcal)$ and $k \in \Cbb^\ast$.
\end{remark}

As in Equation~\eqref{eqn.Zshape}, a $\sigma$-Ptolemy assignment $c \in P^\sigma(\Tcal)$ determines the shape parameters of tetrahedra of $\Tcal$.
These shape parameters automatically form a solution of $\Tcal$, and its holonomy representation admits an $\mathrm{SL}_2(\Cbb)$-lift whose conjugacy class is $\Phi(c) \in \Xcal(M)$.
Conversely, a standard pseudo-developing map argument shows the following (see~\cite[Sec.2.3]{Yoon19} or \cite[Thm.1.12]{GTZ} for details).

\begin{proposition}[\cite{Yoon19}]\label{prop.exist} Let $z_1,\ldots, z_N$  be a solution of $\Tcal$ whose holonomy representation admits an $\mathrm{SL}_2(\Cbb)$-lift $\rho$.
	Then for any cocycle $\sigma : \Ecal(\partial \overline{M})  \rightarrow \{\pm1\}$  satisfying Equation~\eqref{eqn.Compatible} there exists $c \in P^\sigma(\Tcal)$ corresponding to the solution.
\end{proposition}

The geometric representation of a hyperbolic 3-manifold always has an $\mathrm{SL}_2(\Cbb)$-lift \cite[Prop.3.1.1]{CS83}.
We say that a cocycle $\sigma : \Ecal(\partial \overline{M}) \rightarrow \{\pm1\}$ is an \emph{obstruction coycle} if it satisfies Equation~\eqref{eqn.Compatible} for a lift $\rho$ of the geometric representation of $M$.
It follows from Proposition~\ref{prop.exist} that for any obstruction cocycle $\sigma$ there exists $c \in P^\sigma(\Tcal)$ corresponding to the geometric solution of $\Tcal$.

\subsection{An alternative definition of the twisted 1-loop invariant} \label{sec.twistPtolemy}

We now fix an epimorphism $\alpha : \pi_1(M) \twoheadrightarrow \Zbb$.
Let $\widetilde{M}$ be the infinite cyclic cover of $M$ corresponding to the kernel of $\alpha$ and $\widetilde{\Tcal}$ be the ideal triangulation of $\widetilde{M}$ induced from $\Tcal$.
We choose a fundamental domain of $M$ in $\widetilde{M}$ and denote by $\widetilde{e}_i$ and $\widetilde{\Delta}_j$ the lifts of $e_i$ and $\Delta_j$, respectively, to the fundamental domain (recall that $e_1,\ldots,e_N$ and $\Delta_1,\ldots,\Delta_N$ are the edges and the tetrahedra of $\Tcal$).
It is clear that  $\widetilde{e}_1,\ldots, \widetilde{e}_N$ are edges of $\widetilde{\Delta}_1,\ldots, \widetilde{\Delta}_N$ but those are not all of the edges of $\widetilde{\Delta}_1,\ldots, \widetilde{\Delta}_N$. We denote by $\widetilde{e}_{N+1}, \ldots, \widetilde{e}_{N+L}$ the other edges of $\widetilde{\Delta}_1,\ldots, \widetilde{\Delta}_N$.
It follows that for $N+1 \leq i \leq N+L$ there exists a unique pair of $1 \leq j(i) \leq N$ and  $k(i) \in \Zbb$ satisfying
\begin{equation} \label{eqn.UniqueLift}
	\widetilde{e}_{i} = g^{k(i)} \cdot \widetilde{e}_{j(i)}
\end{equation}
where $g$ is a generator of the deck transformation group $\Zbb$ of $\widetilde{M}$.

We choose an orientation of each edge of $\Tcal$ and orient the edges of $\widetilde{\Tcal}$ accordingly.
We then assign a variable $c_i$ to  each $\widetilde{e}_i$ $(1 \leq i \leq N+L$) and 
define an element $p^\sigma_i \in \Zbb[c_1 ,\ldots,c_{N+L}, t^{\pm1}]$ as follows.
\begin{itemize}
	\item let $p^\sigma_i$ be the left-hand side of the $\sigma$-Ptolemy equation~\eqref{eqn.sgPtolemy} of $\widetilde{\Delta}_i$ for $1 \leq i \leq N$;
	\item let $p^\sigma_{i}=  c_{i} - t^{-k(i)} c_{j(i)}$ for $N+1 \leq i \leq N+L$.
\end{itemize}
For $N+1 \leq i \leq N+L$ we often write $p^\sigma_{i}(t)$ instead of $p^\sigma_{i}$ if we need to specify the value of the variable $t$.
For instance, there is an obvious bijection
\[P^\sigma(\Tcal) \simeq \{ (c_1,\ldots,c_{N+L}) \in (\Cbb^\ast)^{N+L} : p^\sigma_1 = \cdots = p^\sigma_N = p^\sigma_{N+1}(1)= \cdots = p^\sigma_{N+L}(1)=0\}\]
sending $c \in P^\sigma(\Tcal)$ to $(c_1,\ldots,c_{N+L})$ where $c_i=c(e_i)$ for $1 \leq i \leq N$ and $c_i=c_{j(i)}$ for $N+1 \leq i \leq N+L$. In particular, if $\sigma$ is an obstruction cocycle, then there exists a point $(c_1,\ldots,c_{N+L}) \in (\Cbb^\ast)^{N+L}$ corresponding to the geometric solution of $\Tcal$.

\begin{theorem} \label{thm.twistPtolemy}
	For an obstruction cocycle $\sigma$	we have	\[\tau^{\CS}(\Tcal,t) = \pm \frac{1}{c_1\cdots c_N}\det \left( \frac{\partial p^\sigma_{i}}{\partial c_j} \right), \quad 1 \leq i,j \leq N+L\]
	where the right-hand side is evaluated at a point $(c_1,\ldots,c_{N+L})$ corresponding to the geometric solution of $\Tcal$.
\end{theorem}

\begin{proof}
	For $N+1 \leq i \leq N+L$ we have
	\begin{equation}
		\frac{\partial p^\sigma_i}{\partial c_j} = 
		\left \{
		\begin{array}{ll}
			1 & \textrm {if } j=i \\
			-t^{-k(i)} & \textrm{if } j=j(i)\\
			0 & \textrm{otherwise}
		\end{array}.
		\right.
	\end{equation}
	Thus adding $t^{-k(i)}$ times the $i$-column of the jacobian matrix $(\partial p^\sigma_i/\partial c_j)$ to the $j(i)$-th column, we have 
	\begin{equation}  \nonumber
		\begin{pmatrix}
			\mathbf{Y} & \rvline &  * \\
			\hline
			\bigzero & \rvline & \mathrm{Id}_L 
		\end{pmatrix}
	\end{equation}
	where $\mathrm{Id}_L$ is the identity matrix of size $L$ and $\mathbf{Y}$ is a square matrix of size $N$ with $\det \mathbf{Y} = \det (\frac{\partial p_i^\sigma}{\partial c_j})$.

	Let $f^\square =(f^\square_1,\ldots, f^\square_N)\in \Zbb^N$ be a combinatorial flattening  of $\Tcal$. We claim that 
	\begin{equation} \label{eqn.keyclaim}
		\begin{array}{ll}
			&\mathrm{diag} (\pm c_1,\ldots, \pm c_N) \, \mathbf{Y}^T\, \mathrm{diag} (
			c^{f_1}_{\widetilde{\Delta}_1},\ldots, c^{f_N}_{\widetilde{\Delta}_N})^{-1} \\[3pt]
			& = \big(\mathbf{G}(t)\, \mathrm{diag}(\zeta)+\mathbf{G}'(t) \,
			\mathrm{diag}(\zeta')+\mathbf{G}''(t) \,
			\mathrm{diag}(\zeta'')	\big) \, \mathrm{diag} (\zeta^{f_1}_{\widetilde{\Delta}_1}, \ldots,\zeta^{f_N}_{\widetilde{\Delta}_N})^{-1}
		\end{array}
	\end{equation}
	where $c^{f_i}_{\widetilde{\Delta}_i}$ and $\zeta^{f_i}_{\widetilde{\Delta}_i}$ are the notations defined in Lemma~\ref{lem.key2}.
	We fix $1 \leq I \leq N$ and let $N+1 \leq I_1,\ldots,I_m \leq N+L$ be the indices satisfying $j(I_1)=\cdots =j(I_m)=I$. 
	This means that the edges $\widetilde{e}_{I_1},\ldots,\widetilde{e}_{I_m}$  are lifts of the edge $e_I$.
	Then 	we have
	\[\mathbf{Y}_{JI} = \frac{\partial p^\sigma_J }{\partial c_{I}} +\frac{\partial p^\sigma_J }{\partial c_{I_1}} t^{-k(I_1)} + \cdots +\frac{\partial p^\sigma_J}{\partial c_{I_m}} t^{-k(I_m)}\] for $1 \leq J \leq N$,  and 	since $c_I = c_{I_1} = \cdots = c_{I_m}$,
	\begin{equation} \label{eqn.lem}
		\frac{c_{I}}{ c^{f_J}_{\widetilde{\Delta}_J}} \mathbf{Y}_{JI} = \frac{c_{I}}{ c^{f_J}_{\widetilde{\Delta}_J}} \frac{\partial p^\sigma_J }{\partial c_{I}} + \frac{c_{I_1}}{ c^{f_J}_{\widetilde{\Delta}_J}}\frac{\partial p^\sigma_J}{\partial c_{I_1}} t^{-k(I_1)} + \cdots  + \frac{c_{I_m}}{ c^{f_J}_{\widetilde{\Delta}_J}}\frac{\partial p^\sigma_J}{\partial c_{I_m}} t^{-k(I_m)}.
	\end{equation}
	A tetrahedron $\Delta_J$ is attached to $e_I$ if and only if $\widetilde{\Delta}_J$ is attached to one of $\widetilde{e}_I, \widetilde{e}_{I_1} ,\ldots,\widetilde{e}_{l_m}$.
	Also, a tetrahedron $\widetilde{\Delta}_J$ is attached to $\widetilde{e}_{I_l}$ for some $1 \leq l \leq m$ if and only if $g^{-k(I_l)} \cdot \widetilde{\Delta}_J$ is attached to $\widetilde{e}_{I}$.
	Combining this fact with Equation~\eqref{eqn.lem}, we obtain Equation~\eqref{eqn.keyclaim} from Lemma~\ref{lem.key2}.
	Therefore, 
	\begin{equation}
		\pm  \frac{c_1\cdots c_N}{c^{f_1}_{\widetilde{\Delta}_1}\cdots c^{f_N}_{\widetilde{\Delta}_N}} \det \mathbf{Y} = \frac{ \det \big(\mathbf{G}(t)\, \mathrm{diag}(\zeta)+\mathbf{G}'(t) \,
			\mathrm{diag}(\zeta')+\mathbf{G}''(t) \,
			\mathrm{diag}(\zeta'')	\big)}{ \zeta^{f_1}_{\widetilde{\Delta}_1} \cdots \zeta^{f_N}_{\widetilde{\Delta}_N}} = \tau^\CS(\Tcal,t).
	\end{equation}
	On the other hand, we have \[{c^{f_1}_{\widetilde{\Delta}_1}}\cdots {c^{f_N}_{\widetilde{\Delta}_N}}=c_1^2 \cdots c_N^2\]
	since $(f,f',f'')$ satisfies $\mathbf{G} f+\mathbf{G}' f' + \mathbf{G}'' f''  =(2,\ldots,2)^T$.
	This completes the proof:
	\begin{equation}
		\tau^\CS(\Tcal,t) = \pm  \frac{c_1\cdots c_N}{{c^{f_1}_{\widetilde{\Delta}_1}}\cdots {c^{f_N}_{\widetilde{\Delta}_N}}} \det \mathbf{Y} = \pm \frac{1}{c_1 \cdots c_N} \det \mathbf{Y} = \pm \frac{1}{c_1 \cdots c_N} \det \left( \frac{\partial p^\sigma_{i}}{\partial c_j} \right).
	\end{equation}
\end{proof}

	 Combining Theorem~\ref{thm.twistPtolemy} with \cite[Thm.1.7]{GY21}, we also obtain an alternative expression of the 1-loop invariant \cite{DG1} in terms of $\sigma$-Ptolemy coordinates:
	\begin{corollary} \label{cor.1loop}
		We have
			\[\tau_\lambda^{\CS}(\Tcal) = \pm \frac{1}{c_1\cdots c_N} \left. \dfrac{d}{dt} \right|_{t=1} \det \left( \frac{\partial p^\sigma_{i}}{\partial c_j} \right), \quad 1 \leq i,j \leq N+L\]	
		where $\lambda$ is the peripheral curve satisfying $\alpha(\lambda)=0$.
	\end{corollary}

\begin{remark} \label{rmk.thirdcondition}
	A combinatorial flattening is required to satisfy the three conditions \eqref{eqn.f1}, \eqref{eqn.f2} and \eqref{eqn.f3}.
	However, we only used the first and second conditions to obtain Theorem~\ref{thm.twistPtolemy} which provides an  expression of the twisted 1-loop invariant without using a flattening. 
	This implies that we actually do not need the third condition~\eqref{eqn.f3} in the first place.
\end{remark}

	\section{Once-punctured torus bundles}	 \label{sec.POTB}

Let $F$ be the once-punctured torus and  $G$ be the mapping class group of  $F$.
For $\varphi \in G$ we denote by $M_\varphi$ the $F$-bundle over the circle $S^1$  whose monodromy is $\varphi$, i.e.,
\[ M_\varphi  = F \times [0,1] /_\sim\]
where the quotient $\sim$ identifies $F \times \{0\}$ with $F \times \{1\}$ by	 a diffeomorphism of $F$ whose isotopy class is $\varphi$. 
It is well-known that $G$ is identified with $\mathrm{SL}_2(\Zbb)$ and that $M_\varphi$ is hyperbolic if $\varphi \in \mathrm{SL}_2(\Zbb)$ has distinct real eigenvalues. 
Also, up to conjugation we can write 
\[ \varphi = \underbrace{R\cdots R}_{s_1}\underbrace{L\cdots L}_{t_1} \cdots \underbrace{R\cdots R}_{s_n}\underbrace{L\cdots L}_{t_n}\]
for some positive integers $s_1,t_1,\ldots, s_n,t_n$ where 
\[R= \begin{pmatrix} 1 & 1 \\ 0 & 1 \end{pmatrix} \textrm{ and }
L=\begin{pmatrix} 1 & 0  \\ 1 & 1 \end{pmatrix}.\]
We denote by $\varphi_i \in \{R,L\}$ the $i$-th letter of $\varphi$ for $1 \leq i \leq N=s_1+t_1+\cdots+s_n+t_n$, and let $\varphi_0 := \varphi_N$.
For the sake of simple exposition, we assume that $t_n \geq 2$. The other case  $t_n=1$ has only minor variants.

\subsection{The canonical triangulation of $M_\varphi$} \label{sec.canonical}
The canonical triangulation $\Tcal_\varphi$ of $M_\varphi$ consists of $N$ ideal tetrahedra.
Letting $M_{0}$ be the once-punctured torus, we inductively construct a manifold $M_i$ $(1\leq i \leq N$) by attaching an ideal tetrahedron $\Delta_i$ to $M_{i-1}$ where the way of attaching $\Delta_i$ depends on whether  $\varphi_{i-1}$ is $R$ or $L$ as in Figure~\ref{fig.LRaction}. 
Then the manifold $M_\varphi$ is obtained from $M_N$ by identifying the remaining faces of the last tetrahedron $\Delta_{N}$ with the initial surface $M_{0}$ in a trivial way.
We refer to \cite{Gue06} for details. 
We denote by $e_i$ the new edge (the bold edge in Figure~\ref{fig.LRaction}) created when we attach $\Delta_i$ to $M_{i-1}$ and by $z_i$ the shape parameter of $\Delta_i$ at $e_i$. 	

\begin{figure}[htpb!]
	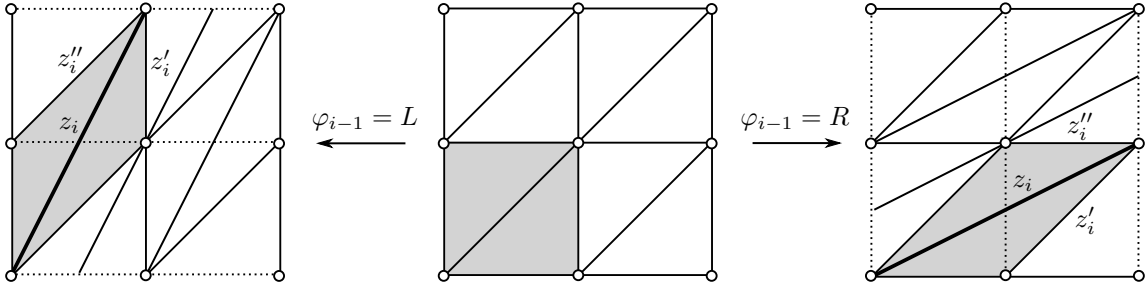
	\caption{Four copies of $F$ and the way of attaching $\Delta_i$.}
	\label{fig.LRaction}
\end{figure}	

The canonical triangulation $\Tcal_\varphi$ induces the triangulation of the peripheral torus of $M_\varphi$ where each  $\Delta_i$ provides four triangles (see Figure \ref{fig.PeripheralTorus}).
We refer to \cite[Sec.4]{Gue06} for a detailed description of this triangulation.
Each edge $e_i$ appears in this triangulation as a vertex and thus 
the gluing equation of $e_i$ can be easily computed from the triangles adjacent to the vertex:
\begin{equation} \label{eqn.GE}
	e_i : \quad
	\left\{
	\begin{array}{ll}
		z_{\overline{i}} \, (z'_{\overline{i+1}})^2 \cdots (z'_{\overline{i+j}})^2 \, z_{\overline{i+j+1}}=1 & \textrm{if } \varphi_{i}=R\\[5pt]
		z_{\overline{i}} \, (z''_{\overline{i+1}})^2 \cdots (z''_{\overline{i+j}})^2 \, z_{\overline{i+j+1}}=1& \textrm{if } \varphi_{i}=L
	\end{array} 
	\right.
\end{equation}
for all $1 \leq i \leq  N$
where $1 \leq \overline{k} \leq N$ denotes the number  congruent to $k \in \Nbb$ in modulo $N$ and $j \geq 1$ is the smallest integer satisfying $\varphi_{\overline{i+j}}= \varphi_{\overline{i}}$.
Remark that we have $i+j+1 \leq N$ so that we actually do not need the overline symbols in Equation \eqref{eqn.GE} for all $1 \leq i \leq N$  but  $i=N-t_n$ (when $\varphi_{i}$ is the last $R$ in $\varphi$) and $i= N-1,N$ (when $\varphi_{i}$ is one of the last two $L$'s in $\varphi$). 
Without using the overline symbol, we have
\begin{equation} \label{eqn.GE1}
	e_i : \quad
	\left\{
	\begin{array}{ll}
		z_{i} \, (z'_{i+1})^2 \cdots (z'_{i+j})^2 \, z_{i+j+1}=1 & \textrm{if } \varphi_{i}=R\\[5pt]
		z_{i} \, (z''_{i+1})^2 \cdots (z''_{i+j})^2 \, z_{i+j+1}=1& \textrm{if } \varphi_{i}=L
	\end{array} 
	\right.
\end{equation}
for all $1 \leq i \leq N$ but $i = N-t_n,  N-1,  N$ where the rest cases are
\begin{equation} \label{eqn.GE2}
	\begin{array}{rl}
		e_{N-t_n}: &\quad	z_{N-t_n} \, (z'_{N-t_n+1})^2 \cdots (z'_{N})^2 \times  (z'_{1})^2 \, z_2=1,\\[5pt]		
		e_{N-1}:&\quad 	z_{N-1} \, (z''_{N})^2 \times z_1=1,\\[5pt]		
		e_{N}:&\quad 	z_{N} \times (z''_{1})^2 (z''_{2})^2\cdots (z''_{t_1+1})^2 \, z_{t_1+2}=1 .\\[5pt]		
	\end{array}
\end{equation}
Here we put $\times$ intentionally to distinguish the indices that come after $N$. This would be helpful for understanding Lemma~\ref{lem.twist} below.
\begin{figure}[htpb!]
	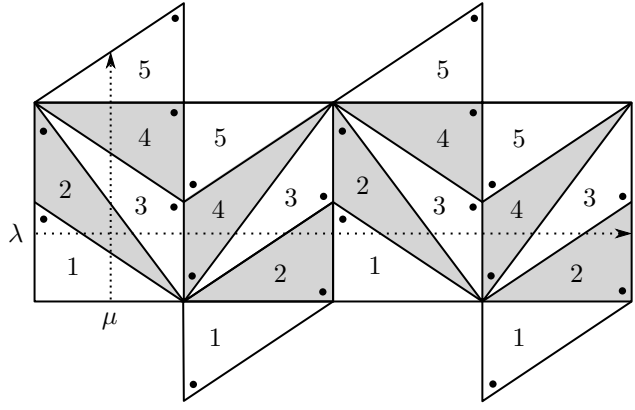
	\caption{The induced triangulation of the peripheral torus for $\varphi = R^2L^3$. The triangles with index $1 \leq i \leq 5$ are from $\Delta_i$, and the dots indicate where the parameter $z_i$ is assigned.}
	\label{fig.PeripheralTorus}
\end{figure}

\begin{example} \label{ex.R2L3}
	Let $\varphi = R^2L^3$.	The canonical triangulation $\Tcal_\varphi$  has 5 edges whose gluing equations are
	\begin{align*}
		e_1: \quad &z_1 (z_2')^2 z_3 =1, \\
		e_2: \quad&z_2 (z_3')^2 (z_4')^2 (z_5')^2 \times (z_1')^2 z_2 =1,\\
		e_3: \quad& z_3 (z''_4)^2 z_5 =1, \\
		e_4: \quad& z_4 (z''_5)^2 \times z_1 =1,\\
		e_5: \quad& z_5 \times (z''_1)^2(z''_2)^2 (z''_3)^2 z_4 =1.
	\end{align*}
	If we choose peripheral curves $\mu$ and $\lambda$ as in Figure \ref{fig.PeripheralTorus}, then their completeness equations are 
	\begin{align}
		\mu: \quad & (z_1'')^{-1} (z_2'')^{-1} z_3'z_4'z_5'  =1, \label{eqn.mu}\\
		\lambda: \quad& \left(z_1 (z_2'')^{-1} (z_3'')^{-1} z_4^{-1} (z_3'')^{-1}z_2'\right)^2 =1. \label{eqn.lambda}
	\end{align}
	Note that $\mu$ and $\lambda$ correspond to the base circle of the bundle $M_\varphi \rightarrow S^1$ and  the puncture of  $F$, respectively and that they generate the fundamental group of the peripheral torus.
	A numerical computation shows that 
	the geometric solution  of $\Tcal_\varphi$ is approximately
	\begin{equation} \label{eqn.sol}
		\begin{array}{ll}
			z_1 \approx -0.19373 + 0.90574 i , &  z_2 \approx 0.80627 + 0.90574  i, \\[3pt]
			z_3 \approx -0.19373 + 0.90574 i, & z_4 \approx 0.35508 + 0.35232 i,\\[3pt]
			z_5 \approx 0.35508 + 0.35232 i &
		\end{array}
	\end{equation}
	with $\mathrm{Vol}(M_\varphi) = \mathrm{Vol}(\Delta_1)+ \cdots+ \mathrm{Vol}(\Delta_5)  \approx4.17775$.
\end{example}

\subsection{The twisted 1-loop invariant of $\Tcal_\varphi$} \label{sec.twistOPT}

Let $\alpha : \pi_1(M_\varphi) \rightarrow \Zbb$  be the epimorphism  induced from the projection map $M_\varphi \rightarrow S^1$ and 
$\widetilde{M}_\varphi$ be the infinite cyclic cover corresponding to the kernel of $\alpha$.
We identify $\widetilde{M}_\varphi$ with  $F \times \Rbb$ and choose a fundamental domain of $M_\varphi$  as $F\times (0,1]$. Recall Section~\ref{sec.twist} that the choice of the fundamental domain determines twisted gluing equation matrices $\mathbf{G}^\square(t)$  of $\Tcal_\varphi$.

\begin{lemma}\label{lem.twist}
	The $i$-th row of $		\mathbf{G}(t)\, \mathrm{diag}(\zeta)+\mathbf{G}'(t) \,
	\mathrm{diag}(\zeta')+\mathbf{G}''(t) \,
	\mathrm{diag}(\zeta'')$ is
	\begin{equation} \label{eqn.row1}
		\left\{
		\begin{array}{ll}
			\begin{pmatrix} 
				0 & \cdots & 0 & \zeta_i & 2\zeta_{i+1}' & \cdots & 2\zeta_{i+j}' & \zeta_{i+j+1} & 0 & \cdots & 0
			\end{pmatrix} & \textrm{if } \varphi_i=R\\[5pt]
			\begin{pmatrix} 
				0 & \cdots & 0 & \zeta_i & 2\zeta_{i+1}'' & \cdots & 2\zeta_{i+j}'' & \zeta_{i+j+1} & 0 & \cdots & 0 
			\end{pmatrix} & \textrm{if } \varphi_i=L\\
		\end{array}
		\right.
	\end{equation} for all $1 \leq i \leq N$ but $i=N-t_n,  N-1,  N$	where $j \geq 1$ is the smallest integer satisfying $\varphi_{i+j}=\varphi_{i}$. 
	The rest three  rows are 
	\begin{equation} \label{eqn.row2}
		\begin{array}{rl}
			\textrm{the $(N-t_n)$-th row :} & 	\begin{pmatrix} 
				0 & \cdots & 0 & \zeta_{N-t_n} & 2\zeta_{N-t_n+1}' & \cdots & 2\zeta_{N}'
			\end{pmatrix} + 
			\begin{pmatrix}
				2 \zeta'_1 & \zeta_2  & 0 & \cdots & 0
			\end{pmatrix} t	 , \\[5pt]
			\textrm{the $(N-1)$-st row :} & 	\begin{pmatrix} 
				0 & \cdots & 0 & \zeta_{N-1} &  2{\zeta_{N}''}
			\end{pmatrix} + 
			\begin{pmatrix}
				\zeta_1  & 0 & \cdots & 0
			\end{pmatrix} t,	\\[5pt]
			\textrm{the $N$-th row :} & 	\begin{pmatrix} 
				0 & \cdots & 0 & \zeta_{N}
			\end{pmatrix} +
			\begin{pmatrix}
				2 \zeta''_1  & \cdots & 2 \zeta''_{t_1+1} &  \zeta_{t_1+2} & 0 & \cdots & 0
			\end{pmatrix} t.
		\end{array}
	\end{equation}	
\end{lemma}
\begin{proof} Recall that	$\widetilde{e}_i$ and $\widetilde{\Delta}_j$ denotes the lifts of $e_i$ and $\Delta_j$ to the fundamental domain, respectively and that  the tetrahedra $\Delta_{\overline{i}},\ldots,\Delta_{\overline{i+j+1}}$ are attached to the edge $e_i$ in $\Tcal_\varphi$ (see Equation~\eqref{eqn.GE}).
	Lifting $e_i$ to $\widetilde{e}_i$, these tetrahedra $\Delta_{\overline{k}}$ ($i \leq k  \leq i+j+1$) are lifted to $g^\frac{k-\overline{k}}{N} \cdot \widetilde{\Delta}_{\overline{k}}$ accordingly. Here $g$ is the generator of the deck transformation group of $\widetilde{M}_\varphi$ sending the fundamental domain $F\times (0,1]$ to $F \times (1,2]$.
	Since we have $i+j+1 \leq N$ for $i \neq N-t_n,  N-1,  N$ as in Equation \eqref{eqn.GE1}, that is, every tetrahedron  attached to $e_i$ is also lifted into the fundamental domain, we obtain Equation~\eqref{eqn.row1} from Remark~\ref{rmk.partial}. Similarly, we obtain Equation~\eqref{eqn.row2} from Equation~\eqref{eqn.GE2}.
\end{proof}

\begin{theorem}\label{thm.X}	
	The twisted 1-loop invariant  of $\Tcal_\varphi$ is given by
	\[\tau^\CS(\Tcal_\varphi, t)  = \det \mathbf{X}\]
	where $\mathbf{X}$ is the $N \times N$ matrix whose $i$-th row $\mathbf{X}_i$ is
	\begin{equation} \label{eqn.Xrow}
		\mathbf{X}_i= \left\{
		\begin{array}{ll}
			\begin{pmatrix} 
				0 & \cdots & 0 & 1 & \frac{2z_{i+1}}{1-z_{i+1}} & \cdots & \frac{2 z_{i+j}}{1-z_{i+j}} & 1 & 0 & \cdots & 0
			\end{pmatrix} & \textrm{if } \varphi_i=R\\[5pt]
			\begin{pmatrix} 
				0 & \cdots & 0 & 1 & \frac{2 }{z_{i+1}-1} & \cdots & \frac{2}{z_{i+j}-1} & 1 & 0 & \cdots & 0
			\end{pmatrix} & \textrm{if } \varphi_i=L\\
		\end{array}
		\right.
	\end{equation}
	for all $1\leq i\leq N$ but $i=N-t_n, N-1, N$ where $j \geq 1$ is the smallest integer satisfying $\varphi_{i+j}=\varphi_{i}$. The rest rows of $\mathbf{X}$ are 
	\begin{equation} \label{eqn.xrest}
		\begin{array}{lcl}
			\mathbf{X}_{N-t_n}&= & 	\begin{pmatrix} 
				0 & \cdots & 0 & 1 & \frac{2 z_{N-t_n+1}}{1-z_{N-t_n+1}} & \cdots & \frac{2 z_{N}}{1-z_{N}}
			\end{pmatrix} + 
			\begin{pmatrix}
				\frac{2 z_{1}}{1-z_{1}}  & 1 & 0 & \cdots & 0
			\end{pmatrix} t	, \\[5pt]
			\mathbf{X}_{N-1}&= & 	\begin{pmatrix} 
				0 & \cdots & 0 & 1  &  \frac{2}{z_{N}-1}
			\end{pmatrix} + 
			\begin{pmatrix}
				1  & 0 & \cdots & 0
			\end{pmatrix} t,	\\[5pt]
			\mathbf{X}_{N}&= & \begin{pmatrix} 
				0 & \cdots & 0 & 1
			\end{pmatrix} +
			\begin{pmatrix}
				\frac{2}{ z_1-1}  & \cdots & \frac{2}{z_{t_1+1}-1} &  1 & 0 & \cdots & 0
			\end{pmatrix} t \, .
		\end{array}
	\end{equation}	
\end{theorem}
\begin{proof} 
		From Lemma~\ref{lem.twist}  with the  fact that  $\zeta'_i/\zeta_i=z_i/(1-z_i)$ and $\zeta''_i/\zeta_i=1/(z_i-1)$
	we have
	\[\mathbf{X} =\big(\mathbf{G}(t)\, \mathrm{diag}(\zeta)+\mathbf{G}'(t) \,
	\mathrm{diag}(\zeta')+\mathbf{G}''(t) \,
	\mathrm{diag}(\zeta'') \big) \, \mathrm{diag}(\zeta_1,\ldots,\zeta_N)^{-1}.\]
	On the other hand, three vectors $f=(1,\ldots,1)^T$ and $f'=f''=(0,\ldots,0)^T$ satisfy the conditions~\eqref{eqn.f1} and~\eqref{eqn.f2}; it is clear that $f+f'+f'' = (1,\ldots,1)^T$, and $\mathbf{G} f+\mathbf{G}' f' + \mathbf{G}'' f'' =\mathbf{G} f=(2,\ldots,2)^T$ since the gluing equation \eqref{eqn.GE} of each edge $e_i$ has exactly two shape parameters $z_{\overline{i}}$ and $z_{\overline{i+j+1}}$ of the form $z$. 
	Then it follows from Remark~\ref{rmk.thirdcondition} that 
	\[ \tau^\CS(\Tcal_\varphi, t) = \frac{\det\big(
		\mathbf{G}(t)\, \mathrm{diag}(\zeta)+\mathbf{G}'(t) \,
		\mathrm{diag}(\zeta')+\mathbf{G}''(t) \,
		\mathrm{diag}(\zeta'') \big)}{\zeta_1 \cdots \zeta_N} = \det \mathbf{X} \, .\]
\end{proof}

\begin{example}[Continued] \label{ex.twist}
	Applying the case of $\varphi = R^2L^3$ to Theorem \ref{thm.X}, we have
	\begin{equation} \label{eqn.Xample}
		\mathbf{X}=\begin{pmatrix}
			1 & \frac{2z_2}{1-z_2} & 1 & 0 & 0 \\
			0 & 1 &  \frac{2z_3}{1-z_3} & \frac{2z_4}{1-z_4} & \frac{2z_5}{1-z_5} \\
			0 & 0 & 1 &  \frac{2}{z_4-1} & 1 \\
			0 & 0 & 0 & 1 & \frac{2}{z_5-1} \\
			0 & 0 & 0 & 0 & 1 
		\end{pmatrix} + 
		\begin{pmatrix}
			0 & 0 & 0 & 0 & 0 \\
			\frac{2z_1}{1-z_1} &  1 & 0 & 0 & 0 \\
			0 & 0 & 0 & 0 & 0 \\
			1 & 0 & 0 & 0 & 0 \\
			\frac{2}{z_1-1} & \frac{2}{z_2-1} & \frac{2}{z_3-1} & 1 & 0 
		\end{pmatrix} t \, .
	\end{equation}
	An easy exercise is that $\tau^\CS(\Tcal_\varphi,t)=\det \mathbf{X}$ is a polynomial of degree $3$ whose leading coefficient and constant term are $-1$ and $1$, respectively. Also, since $\tau^\CS(\Tcal_\varphi,1)=0$ \cite[Thm.1.7]{GY21},  we have $\tau^\CS(\Tcal_\varphi,t)=1 - \alpha t + \alpha t^2 -t^3$ for some $\alpha \in \Cbb$.
	Indeed, substituting the geometric solution~\eqref{eqn.sol} to $\det \mathbf{X} $, we obtain
	\begin{equation} \label{eqn.extwist}
		\tau^\CS(\Tcal_\varphi,t) \approx 1.00000  - (31.45667 + 9.44217i)\, t + (31.45667 + 9.44217 i)\, t^2 - 1.00000 t^3. 
	\end{equation}
\end{example}

	\subsection{$\sigma$-Ptolemy assignments of $\Tcal_\varphi$} \label{sec.PtolemyOncePun}
%		In this section, we compute $\sigma$-Ptolemy assignments of the canonical triangulation $\Tcal_\varphi$ and present a formula for the adjoint twisted Alexander polynomial of $M_\varphi$ in terms of $\sigma$-Ptolemy assignment.
%		For simplicity we confuse the canonical triangulation $\Tcal_\varphi$ with its underlying space $M_\varphi$.
		
	We start with 
	the once-punctured torus $F$ consisting of two ideal triangles. Chopping off the puncture of $F$, we obtain the compact surface $\overline{F}$ consisted of two truncated triangles. 
	We choose  an assignment $\sigma : \Ecal(\partial \overline{F}) \rightarrow \{\pm1\}$ as in Figure~\ref{fig.puncturedF} (note that the orientations of short-edges do not matter as the target group is $\{\pm1\}$).
	Since we do not have any tetrahedron so far, a $\sigma$-Ptolemy assignment $c : \Ecal(F) \rightarrow \Cbb^\ast$ is an assignment simply satisfying $c(-e)=-c(e)$ for $e \in \Ecal(F)$.   
	Choosing the orientations of the edges $e_{N+1}, \, e_{N+2}$ and $e_{N+3}$ of $F$ as in Figure~\ref{fig.puncturedF},  the set $P^\sigma(F)$ of all $\sigma$-Ptolemy assignments of $F$ is identified with $(\Cbb^\ast)^3$:
\[P^\sigma(F) \simeq (\Cbb^\ast)^3, \quad c \mapsto (c(e_{N+1}), \, c(e_{N+2}), \, c(e_{N+3})).\] 
Here we use the index starting from $N+1$ to match it with  the ones used in Sections~\ref{sec.twistPtolemy} and~\ref{sec.canonical}.

	\begin{figure}[htpb!]
		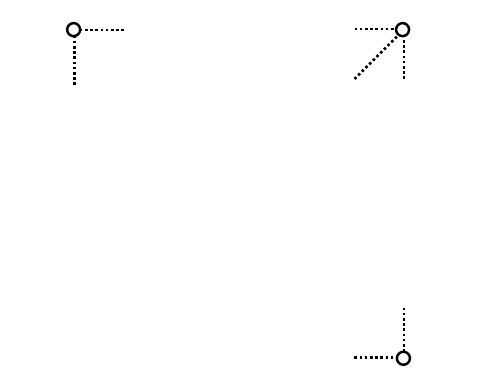
		\caption{The compact surface $\overline{F}$.}
		\label{fig.puncturedF}
	\end{figure}	

	Recall that  a $\sigma$-Ptolemy assignment $c : \Ecal(F) \rightarrow \Cbb^\ast$ determines a cocycle $\rho : \Ecal(\overline{F}) \rightarrow \mathrm{SL}_2(\Cbb)$, and thus 
	there is  a map  $\Phi$ from $P^\sigma(F)$ to  the $\mathrm{SL}_2(\Cbb)$-character variety $\Xcal(F)$ of $F$.
	
	\begin{lemma} \label{prop.algF}
		 The image of $\Phi$ is contained in $\mathrm{tr}_\lambda^{-1}(-2)$ where $\mathrm{tr}_\lambda : \Xcal(F) \rightarrow \Cbb$ is the trace function of  $\lambda = \partial \overline{F}$.
	\end{lemma}

	\begin{proof}
		The boundary curve $\lambda = \partial \overline{F}$ consists of 6 short-edges and the product of signs that $\sigma$ assigns to them is $-1$. It thus follows from 
		Equation~\eqref{eqn.Passign} that the the image of $\lambda$ under the representation $\rho$ is 
		\[ \rho(\lambda) = \begin{pmatrix} -1 & \ast \\ 0 & -1 \end{pmatrix} \]
		up to conjugation, hence $\mathrm{tr} \rho(\lambda)=-2$.
	\end{proof}

	\begin{remark} \label{rmk.coord}
		The fundamental group of $F$ is the free group with two generators, say $g_1$ and $g_2$, and the character variety $\Xcal(F)$  is identified with $\Cbb^3$ with coordinates $\tr_{g_1}$, $\tr_{g_2}$ and $\tr_{g_1 g_2}$.  Using Equation \eqref{eqn.Passign}, we explicitly compute that
		\begin{equation}  \nonumber
			\rho(g_1) = \begin{pmatrix}
				\frac{x^2 + y^2}{x z} & \frac{y}{x^2} \\[4pt]
				y & \frac{z}{x}, 
			\end{pmatrix}, \quad
			\rho(g_2)= \begin{pmatrix}
			0 & -\frac{1}{x} \\[4pt]
			x & - \frac{x^2 + y^2+z^2}{yz}
		\end{pmatrix}
		\end{equation}
	with
		\begin{equation}  \nonumber
			\tr \rho(g_1) = \frac{x^2 +y^2 + z^2}{x z}, \quad 
			\tr \rho(g_2)=  -\frac{x^2 +y^2 + z^2}{y z}, \quad		
			\tr \rho(g_1 g_2)= - \frac{x^2 +y^2 + z^2}{x y} \,.
		\end{equation}
		where $(x,y,z)=(c(e_{N+1}),c(e_{N+2}),c(e_{N+3}))$.
	Also, one can double-check that $\tr \rho(\lambda) = \tr \rho(g_1g_2 g_1^{-1} g_2^{-1}) = -2$ and Remark~\ref{rmk.scaling}.
	\end{remark}

	Recall that  the manifold $M_i$ ($1 \leq i \leq N$) is obtained from $M_{i-1}$  by attaching an ideal tetrahedron $\Delta_i$ to $M_{i-1}$ as in Figure~\ref{fig.LRaction} ($M_0$ is the once-punctured torus $F$).
	For simplicity we confuse the ideal triangulation of $M_i$ with its underlying space $M_i$.
	 We denote by  $\overline{M}_i$ the compact manifold obtained from $M_i$ by replacing $\Delta_1,\ldots, \Delta_i$ to their truncations $\overline{\Delta}_1,\ldots, \overline{\Delta}_i$.
	
	We first extend $\sigma : \Ecal(\partial \overline{M}_{i-1}) \rightarrow \{ \pm1\}$ to $\partial \overline{M}_i$ so that the extension $\sigma :  \Ecal(\partial \overline{M}_i) \rightarrow \{ \pm1\}$ (also denoted by $\sigma$) satisfies the cocycle conditions.
	Explicitly, the truncated tetrahedron $\overline{\Delta}_i$ has 12 short-edges, 6 of which are contained in $\overline{M}_{i-1}$. Thus there are 6 new short-edges when we attach $\overline{\Delta}_i$ to $\overline{M}_{i-1}$, and  we assign $+1$ to all of these new short-edges but one; assign $-1$ to the one in the bottom-left as in Figure \ref{fig.LRPtolemy} (see also Figure \ref{fig.puncturedF}).
	One easily checks that $\sigma :  \Ecal(\partial \overline{M}_i) \rightarrow \{ \pm1\}$  is indeed a cocycle
  (and such an extension is unique).
	
	\begin{figure}[htpb!]
		\input{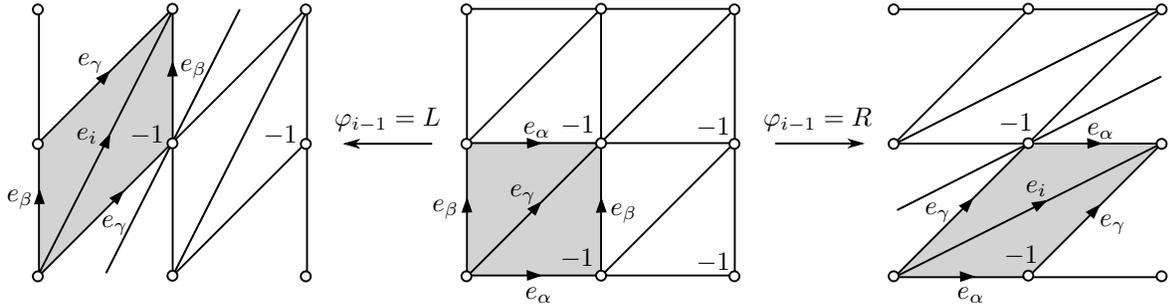}
		\caption{Four copies of $F$ and the way of attaching $\Delta_i$.}
		\label{fig.LRPtolemy}
	\end{figure}

	We then extend a $\sigma$-Ptolemy assignment $c:\Ecal(M_{i-1}) \rightarrow \Cbb^\ast$ to $c:\Ecal(M_i) \rightarrow \Cbb^\ast$ (also denoted by $c$) so that the $\sigma$-Ptolemy equation of $\Delta_i$ is satisfied.
	Explicitly, if we orient the edges $e_\alpha, e_\beta, e_\gamma$ and $e_i$  of  $\Delta_i$ as in Figure~\ref{fig.LRPtolemy}, then the $\sigma$-Ptolemy equation of $\Delta_i$ is  
	\begin{equation} \label{eqn.Ptolemyr}
	\left\{
	\begin{array}{ll}
		c(e_i) c(e_\beta) - c(e_\alpha)^2-c(e_\gamma)^2=0 & \textrm{if } \varphi_{i-1}=R\\[10pt]
		c(e_i) c(e_\alpha) - c(e_\beta)^2 - c(e_\gamma)^2 =0 & \textrm{if } \varphi_{i-1}=L\\
	\end{array} 
	\right.  
\end{equation}	
	and therefore  
	\begin{equation} \label{eqn.PtolemyI}
		c(e_i) =\left\{
		\begin{array}{ll}
			\dfrac{c(e_\alpha)^2+c(e_\gamma)^2}{c(e_\beta)} & \textrm{if } \varphi_{i-1}=R\\[10pt]
			\dfrac{c(e_\beta)^2+c(e_\gamma)^2}{c(e_\alpha)} & \textrm{if } \varphi_{i-1}=L\\
		\end{array}
		\right.  .
	\end{equation}	
	Note that (as in Section \ref{sec.canonical}) $e_i$ is the new edge created when we attach $\Delta_i$ to $M_{i-1}$  and that $c(e_i)$, $1 \leq i\leq N$ is a rational function in the three initial variables $c(e_{N+1}), \, c(e_{N+2})$ and $c(e_{N+3})$.
	
	\begin{remark} Recall  that
			a $\sigma$-Ptolemy assignment determines the shapes of tetrehedra. Precisely, for Figure~\ref{fig.LRPtolemy} the shape parameter $z_i$ of $\Delta_i$ at the edge $e_i$ is given by
		\begin{equation} \label{eqn.shape}
			z_i= \left\{
			\begin{array}{ll}
				- c(e_\gamma)^2/ c(e_\alpha)^2 & \textrm{if } \varphi_{i-1}=R\\[5pt]
				- c(e_\beta)^2 / c(e_\gamma)^2 & \textrm{if } \varphi_{i-1}=L\\
			\end{array}
			\right. .
		\end{equation}	
	\end{remark}	

	The boundary of $M_i$ ($1 \leq i\leq N$) consists of four ideal triangles. We regard two of them forming the initial punctured torus $M_0$ as  a \emph{negative boundary} and the other two triangles as a \emph{positive boundary}, so that the manifold $M_\varphi$ is obtained  by identifying the negative and positive boundaries of $M_N$ in a trivial way. 
	It is clear from Figure~\ref{fig.LRPtolemy} that $\sigma$ assigns the signs for the negative and positive boundaries of $M_i$ in the same fashion. It follows that the cocycle  $\sigma : \Ecal(\partial \overline{M}_N)\rightarrow \{\pm1\}$ induces $\sigma : \Ecal(\partial \overline{M}_\varphi)\rightarrow \{\pm1\}$ (also denoted by $\sigma$).
	On the other hand, the triple of horizontal, vertical, and diagonal edges of the negative boundary of  $M_i$ is $(e_{N+1},\, e_{N+2},\, e_{N+3})$.  See Figure~\ref{fig.puncturedF}.
	Such triples for positive boundaries depend on $i$, as we attach $\Delta_i$ to the positive boundary of $M_{i-1}$ as in Figure~\ref{fig.LRPtolemy}:
	\begin{equation} \label{eqn.change}
		\begin{array}{ll}
		(e_\alpha, \, e_\beta, \, e_\gamma) \mapsto (e_\alpha, \, e_\gamma, \, e_i) & \quad \textrm{if } \varphi_{i-1}= R, \\[2pt]
		(e_\alpha, \, e_\beta, \, e_\gamma) \mapsto (e_\gamma, \, e_\beta, \, e_i) & \quad \textrm{if } \varphi_{i-1} = L. 
	\end{array}
	\end{equation}
	It follows that the triple for the positive boundary of $M_N$ is $(e_{N-1},\, e_{N-t_n}, \, e_N)$. Therefore, a $\sigma$-Ptolemy assignment $c : \Ecal(M_N) \rightarrow \Cbb^\ast$ is required to satisfy
	\begin{equation} \label{eqn.id}
		c(e_{N+1}) = c(e_{N-1}), \quad c(e_{N+2}) =c(e_{N-t_n}), \quad c(e_{N+3}) = c(e_N)
	\end{equation}		
	in order to become a $\sigma$-Ptolemy assignment of $\Tcal_\varphi$.
	Summing up the above, we obtain:
	
	\begin{proposition}
		A $\sigma$-Ptolemy assignment $c \in P^\sigma(\Tcal_\varphi) $ is represented by a point $(c(e_{1}),\ldots,c(e_{N+3})) \in (\Cbb^\ast)^{N+3}$ satisfying Equation~\eqref{eqn.PtolemyI} for all $1 \leq i \leq N$ and Equation~\eqref{eqn.id}.
	\end{proposition}

	\begin{proposition} There exists a $\sigma$-Ptolemy assignment $c \in P^\sigma(\Tcal_\varphi)$ corresponding to the geometric solution of $\Tcal_\varphi$.
	\end{proposition}	
	\begin{proof} 
		It suffices to show that the cocycle $\sigma$ is an obstruction cocycle (see Proposition~\ref{prop.exist}).
		Let $\rho$ be an $\mathrm{SL}_2(\Cbb)$-lift of the geometric representation of $M_\varphi$.
		 The fundamental group of $\pi_1(M_\varphi)$ has a presentation 
		\[ \left\langle g_1,g_2, \mu \, | \, \mu g_1 \mu^{-1} = \varphi_\ast(g_1),\  \mu g_2 \mu^{-1} = \varphi_\ast(g_2) \right\rangle \]
		where $\mu$ is the base circle of the bundle $M_\varphi \rightarrow S^1$ and $\varphi_\ast : \pi_1(F) \rightarrow \pi_1(F)$ is the isomorphism induced from the monodromy $\varphi$.
		This implies that we may assume that $\tr \rho(\mu)=2$; if $\tr \rho(\mu)=-2$, then we consider another lift $\overline{\rho}$ defined by $\overline{\rho}(g_1)={\rho}(g_1)$, $\overline{\rho}(g_2)={\rho}(g_2)$ and $\overline{\rho}(\mu)=-{\rho}(\mu)$.
		On the other hand,  the geometric representation does not admit an $\mathrm{SL}_2(\Cbb)$-lift which is boundary-unipotent \cite{Cal06}.
		Hence, we have $\tr {\rho}(\lambda)=-2$ and 
		\begin{equation}\label{eqn.compatible} 	
		{\rho}(\mu) = \begin{pmatrix} \sigma(\mu) & \ast \\ 0 & \sigma(\mu)^{-1} \end{pmatrix}, \quad
		{\rho}(\lambda) = \begin{pmatrix} \sigma(\lambda) & \ast \\ 0 & \sigma(\lambda)^{-1} \end{pmatrix}
		\end{equation}
		up to conjugation. This proves that $\sigma$ is an obstruction cocycle.

	\end{proof}

%	\begin{remark} 
%		The obstruction for lifting a boundary-parabolic $\mathrm{PSL}_2(\Cbb)$-representation to a boundary-unipotent $\mathrm{SL}_2(\Cbb)$-representation is a class in $H^2(\overline{M}_\varphi, \partial \overline{M}_\varphi;\{\pm1\}) \simeq H_1(\overline$.
%		
%		
%		The cocycle $\sigma : \Ecal(\overline{M}_\varphi) \rightarrow \{\pm1\}$ induces the homomoprhism $\pi_1(M)\rightarrow \{ \pm 1\}$ which is T 
%	\end{remark}

	\begin{example}[Continued]\label{ex.Ptolemy}
			For $\varphi = R^2L^3$ the set $P^\sigma(\Tcal_\varphi)$  consists of points $(c(e_{1}), \ldots, c(e_{8})) \in (\Cbb^\ast)^8$ satisfying
			\begin{equation} \label{eqn.exP1}
				\begin{array}{lllllll}
					& c(e_1)  = \dfrac{c(e_{7})^2 + c(e_{8})^2}{c(e_{6})}, & c(e_2) = \dfrac{c(e_8)^2 + c(e_1)^2}{c(e_{7})},
					& c(e_3) = \dfrac{c(e_{8})^2 + c(e_2)^2}{c(e_1)}, \\[7pt]
					& c(e_4) = \dfrac{c(e_2)^2+c(e_3)^2}{c(e_8)}, &c(e_5)  = \dfrac{c(e_2)^2 + c(e_4)^2}{c(e_3)},
				\end{array}
			\end{equation}
			and 
			\begin{equation} \label{eqn.exP2}
				c(e_{6}) = c(e_4), \quad c(e_{7}) = c(e_2), \quad  c(e_8) = c(e_5)\,.
			\end{equation}
			As $P^\sigma(\Tcal_\varphi)$ admits the multiplication action by elements of $\Cbb^\ast$ (see Remark~\ref{rmk.scaling}), we may assume that $c(e_{6})=1$.
			Then there are finitely many solutions to Equations \eqref{eqn.exP1} and \eqref{eqn.exP2}. Among them, we  find a solution corresponding to the geometric solution~\eqref{eqn.sol} by using Equation~\eqref{eqn.shape}:	 the shape parameters of $\Delta_1,\ldots,\Delta_5$ are
			\begin{equation} \label{eqn.exShp}
				( z_1, z_2,z_3,z_4,z_5) = \left( - \frac{c(e_{7})^2}{c(e_8)^2}, \, - \frac{c(e_1)^2}{c(e_8)^2}, \, - \frac{c(e_2)^2}{c(e_8)^2}, \, 
				- \frac{c(e_2)^2}{c(e_3)^2},\,  - \frac{c(e_2)^2}{c(e_4)^2}\right).
			\end{equation}
			A numerical computation shows that the three initial variables corresponding to the geometric solution~\eqref{eqn.sol} are
			\begin{equation} \label{eqn.geomptol}
					(c(e_{6}), \, c(e_{7}), \, c(e_8)) \approx (1.00000, \, 0.26938 - 0.65395 i, \, 0.64492 - 0.35232i)
			\end{equation}
			with
			\[(c(e_1),\ldots,c(e_5)) \approx ( -0.06329 - 0.80677i, \, 0.26938 - 0.65395 i, \, 1.00000, \, 1.00000, \, 0.64492 - 0.35232 i).\]
	\end{example}

	\subsection{Adjoint twisted Alexander polynomials} \label{sec.ATAP}
	
	We now fix a $\sigma$-Ptolemy assignment $c \in P^\sigma(\Tcal_\varphi)$  and a representation $\rho:\pi_1(M_\varphi) \rightarrow \mathrm{SL}_2(\Cbb)$ determined by $c$ (up to conjugation). 
	We denote by $H^\ast(M_\varphi;V)$ the twisted cohomology of $M_\varphi$ 	where the coefficient $V=\sl \otimes \Cbb[t^{\pm1}]$ is twisted by $\mathrm{Ad}\rho \otimes \alpha$. 
	Here $\mathrm{Ad}\rho$ denotes the adjoint action associated to $\rho$.
	\begin{theorem} \label{thm.ATAP}
		The adjoint twisted Alexander polynomial $\tau(\rho,t)$ associated to $\rho$ is given by
	\begin{equation} \label{eqn.ATAP}
		\tau(\rho,t) = \det \left( t I - \dfrac{ \partial (c(e_{N-1}), c(e_{N-t_n}), c(e_N))}{ \partial(c(e_{N+1}), c(e_{N+2}),c(e_{N+3}))} \right)
	\end{equation}
	if $H^\ast(M_\varphi;V)$ is trivial and the restriction of $\rho$ to the  punctured torus $F$ is irreducible.
	Here $c(e_{N-t_n}),c(e_{N-1})$ and $c(e_N)$ are viewed as (rational) functions in $c(e_{N+1}), c(e_{N+2})$ and $c(e_{N+3})$.
	\end{theorem}
	\begin{proof} 
		As the restriction $\eta$ of $\rho$ to $F$ is irreducible, we have $H^0(F;V)=H^0(F;\sl)=0$. Hence the cohomology $H^i(F;V)$ is non-trivial only for $i=1$.
		Then a routine computation of the Mayer-Vietoris sequence for a bundle over the circle $S^1$ (see e.g. \cite{Fried86}) shows that
		\[\tau(\rho,t)  = \det (t I - \varphi^\ast)\]
		where $\varphi^\ast : H^1(F; V) \rightarrow H^1(F;V)$ is the isomorphism induced from the monodromy $\varphi$. 
		As the restriction of $\alpha$ to $F$ is trivial, we may reduce the coefficient $V$ to $\sl$. Also, it is known that $H^1(F;\sl)$ is isomorphic to the Zariski tangent space $T_{[\eta]} \Xcal(F)$ of the character variety $\Xcal(F)$ at the conjugacy class of $\eta$ (see e.g. \cite{Wei64,Sik12}). Hence we may regard $\varphi^\ast$ as an automorphism of $ T_{[\eta]} \Xcal(F)$.		
		We now consider the differential  $d \tr_\lambda: T_{[\eta]} \Xcal(F) \rightarrow \Cbb$ of the trace function $\tr_\lambda : \Xcal(F) \rightarrow \Cbb$  of $\lambda$ at $[\eta]$.
		Since $[\eta] \in \tr_\lambda^{-1}(-2) \subset \Xcal(F)$
		(see Lemma~\ref{prop.algF}), the kernel of $d \tr_\lambda$ is equal to $T_{[\eta]} (\tr_\lambda^{-1}(-2))$.
		On the other hand, the monodromy $\varphi$  induces an action on $\Xcal(F)$ fixing $[\eta]$ and preserving $\tr_\lambda$ (since $\varphi$ preserves the boundary curve $\lambda$). It follows that $\varphi^\ast : T_{[\eta]} \Xcal(F) \rightarrow T_{[\eta]} \Xcal(F)$ has a matrix form 
		\[
		\varphi^\ast 	=
		\begin{pmatrix}
			&  & \rvline & * \\
			\multicolumn{2}{c}{\smash{\raisebox{.5\normalbaselineskip}{$A$}}}   & \rvline & * \\
		      \hline \\[-\normalbaselineskip]
			0 & 0 & \rvline & 1
		\end{pmatrix}
		\]		
		where the matrix $A$ represents the restriction  of $\varphi^\ast$ to $T_{[\eta]}(\tr_\lambda^{-1}(-2))$

		Let $F_0$ and $F_1$ be the negative and positive boundaries of $M_N$, respectively, so that $M_\varphi$ is obtained by identifying $F_0$ with $F_1$ in a trivial way.	Recall Lemma~\ref{prop.algF} and Remark~\ref{rmk.scaling} (see also Remark~\ref{rmk.coord}) that  $P^\sigma(F_i)/\Cbb^\ast \simeq (\Cbb^\ast)^3 / \Cbb^\ast$ gives a local coordinates of  $\mathrm{tr}^{-1}_\lambda(-2) \subset \Xcal(F)$ for $i=1,2$.
		Precisely, two homogeneous coordinates 
		\[[c(e_{N+1}) : c(e_{N+2}) : c(e_{N+3})] \quad \textrm{and} \quad [c(e_{N-1}): c(e_{N-t_n}):c(e_N)]\] are local coordinates of  $\mathrm{tr}^{-1}_\lambda(-2)$ at $[\eta]$ where the notation $[x:y:z]$ means that we identify it with $[kx:ky:kz]$ for all $k\in \Cbb^\ast$. 
		 Also, recall Equation \eqref{eqn.PtolemyI} that if we multiply the initial variables $c(e_{N+1}), \, c(e_{N+2})$ and $c(e_{N+3})$ by $k \in \Cbb^\ast$, then the rational function $c(e_i)$ becomes $k \cdot c(e_i)$ for all $1 \leq i \leq N$. It implies that $ \frac{\partial (c(e_{N-1}), c(e_{N-t_n}), c(e_N))}{ \partial(c(e_{N+1}), c(e_{N+2}),c(e_{N+3}))}$ has an eigenvalue $1$ and thus
	 	\[ \det \left( t I -	 \dfrac{ \partial (c(e_{N-1}), c(e_{N-t_n}), c(e_N))}{ \partial(c(e_{N+1}), c(e_{N+2}),c(e_{N+3}))} \right)=(t-1) \det (t I -A).\]
	 	This completes the proof, since  $\tau(\rho,t) = \det (t I - \varphi^\ast)=(t-1) \det(tI-A)$.
	\end{proof}

	If $\rho$ is an $\mathrm{SL}_2(\Cbb)$-lift of the geometric representation of $M_\varphi$,  we write $\tau(M_\varphi,t)$ instead of $\tau(\rho,t)$. In this case, 
	it is proved in  \cite[Lem.2.5]{BDHP19} that the twisted cohomology $H^\ast(M_\varphi;V)$ is trivial. Also, the restriction of $\rho$ to $F$ is irreducible;  otherwise, the generators $g_1$ and $g_2$ of $\pi_1(F)$ satisfy  up to conjugation
	\[\rho(g_1) 
		= \begin{pmatrix} m_1 & \ast \\ 0 & m_1^{-1} 
		\end{pmatrix}, \quad
		\rho(g_1) 
		= \begin{pmatrix} m_2 & \ast \\ 0 & m_2^{-1} 
	\end{pmatrix}
	\]  for some $m_1$ and $m_2 \in \Cbb^\ast$, hence $\tr \rho(\lambda) = \tr \rho(g_1 g_2 g_1^{-1} g_2^{-1}) = 2$ which contradicts to the fact that $\tr \rho(\lambda)=-2$.
	
%	\begin{figure}[htpb!]
%		\input{figures/Sequence.pdf_tex}
%		\caption{The sequence of indices.}
%		\label{fig.Sequence}
%	\end{figure}	

	\begin{example}[Continued] From Equation~\eqref{eqn.exP1} we have
	\begin{align*}
		c(e_2) & = 
		\frac{x^2 z^2+\left(y ^2+z ^2\right)^2}{x^2 y },\\[5pt]
		c(e_4) &= \frac{x^8 z ^4+x^6 \left(5 y^2 z^4+4 z^6\right)+2 x^4 z^2 \left(y^2+z^2\right)^2 \left(2 y^2+3 z^2\right)+x^2 \left( y^2+z^2\right)^4 \left(y^2+4 z^2\right)+\left(y^2+z^2\right)^6}{x^6 y^4 z}, \\[5pt]
		c(e_5) & = \frac {x ^{12} z ^6+x^{10}  \left(8 y^2 z^6+6 z^8\right)+x^8 z^4 \left(y^2+z^2\right) \left(4 y^2+3 z^2\right) \left(2 y^2+5 z^2\right)+4 x^6 z^2 \left(y ^2+z^2\right)^4 \left(y^2+5 z^2\right)}{x^9 y^6 z^2} \\
		&   \quad + \frac{x^4 \left(y^2+z^2\right)^5 \left(11y^2 z^2+y^4+15 z^4\right)+2 x^2 \left(y^2+z^2\right)^7 \left(y^2+3 z^2\right)+\left(y^2+ z^2\right)^9}{x^9y^6 z^2}.
		\end{align*}
			where $(x,y,z)=(c(e_{6}),c(e_{7}),c(e_{8}))$.
	Substituting the geometric solution~\eqref{eqn.geomptol} to Theorem \ref{thm.ATAP}
	\[\tau(M_\varphi,t)=\det \left( t I - \frac{ \partial (c(e_{4}), c(e_{2}), c(e_5))}{ \partial(c(e_{6}), c(e_{7}),c(e_8))} \right),\]
	we obtain 
	\[ \tau(M_\varphi,t) \approx 1.00000  - (31.45667 + 9.44217 i)\, t + (31.45667 + 9.44217 i)\, t^2 - 1.00000 t^3\]
	which agrees with the twisted 1-loop invariant $\tau^\CS(\Tcal_\varphi,t)$  computed in Example~\ref{ex.twist}.
	
	\end{example}
 	
	\section{Proofs} \label{sec.proof}
	\subsection{Proof of Theorem~\ref{thm.main2}}  \label{sec.proof1}
	
	Applying the computation in Section~\ref{sec.PtolemyOncePun} to Theorem~\ref{thm.twistPtolemy}, we have
	\begin{equation}
		\tau^\CS(\Tcal_\varphi,t) = \pm \frac{1}{c_1\cdots c_N} \det \left( \dfrac{\partial p^\sigma_i}{\partial c_j} \right), \quad 1 \leq i,j \leq N+3
	\end{equation}
	where $p^\sigma_i$ ($1 \leq i \leq N$) is given by (cf. Equation~\eqref{eqn.Ptolemyr})
		\begin{equation} \label{eqn.PtolemyC}
		p^\sigma_i= \left\{
		\begin{array}{ll}
			c_i c_\beta - c_\alpha^2-c_\gamma^2 & \textrm{if } \varphi_{i-1}=R\\[10pt]
			c_i c_\alpha - c_\beta^2 - c_\gamma^2  & \textrm{if } \varphi_{i-1}=L\\
		\end{array} 
		\right.  \quad \textrm{ for Figure~\ref{fig.LRPtolemy}}
	\end{equation}	
	and (cf. Equation~\eqref{eqn.id})
	\begin{equation}  \nonumber
	p_{N+1}  = c_{N+1} -  c_{N-1} t, \quad 
	p_{N+2} = c_{N+2} - c_{N-t_n} t,  \quad
	p_{N+3} = c_{N+3} -  c_N t.
\end{equation}
Recall that the quad type of $\Delta$ is given by $(e_i,e_{i'})$ for some $i'$. Precisely, for Figure~\ref{fig.LRPtolemy} we have $e_{i'}= e_\beta$ if $\varphi_{i-1}=R$ and $e_{i'}=e_\alpha$ if $\varphi_{i-1}=L$.
Since $f=(1,\ldots,1)^T$ and $f'=f''=(0,\ldots,0)^T$ satisfy $\mathbf{G} f+\mathbf{G}' f' + \mathbf{G}'' f'' =(2,\ldots,2)^T$ (see the proof of Theorem~\ref{thm.X}), we have
\[\prod_{i=1}^N c_i c_{i'}= \prod_{i=1}^N c_{\Delta_i}^{f_i}= \prod_{i=1}^N c^2_i .\]
Here $c_{\Delta_i}^{f_i} $ is the notation defined in Lemma~\ref{lem.key2}.
 It follows that $\prod_i c_{i'}= \prod_i c_i$ and thus
	\begin{equation}
	\tau^\CS(\Tcal_\varphi,t) = \pm \det \left( \dfrac{\partial (\overline{p}^\sigma_1,\ldots,\overline{p}^\sigma_N,\, p^\sigma_{N+1},\, p^\sigma_{N+2},\, 
	p^\sigma_{N+3})}{\partial (c_1,\ldots,c_N,\, c_{N+1},\, c_{N+2},\, c_{N+3})} \right)
\end{equation}
	where	
			\begin{equation} \label{eqn.PtolemyB}
		\overline{p}^\sigma_i= \left\{
		\begin{array}{ll}
			c_i - \dfrac{c_\alpha^2-c_\gamma^2}{ c_\beta} & \textrm{if } \varphi_{i-1}=R\\[10pt]
			c_i - \dfrac{c_\beta^2 - c_\gamma^2}{c_\alpha} & \textrm{if } \varphi_{i-1}=L\\
		\end{array} 
		\right.  \quad \textrm{ for Figure~\ref{fig.LRPtolemy}}.
	\end{equation}	
	
%	\begin{equation}
%	\left\{
%	\begin{array}{rcl}
%		p_{-2} & := &c_{-2} - c_{N-1} \, t \, ,\\
%		p_{-1}&:=&c_{-1} - c_{N-t_n} \, t \, ,\\
%		p_{0}&:=&c_0 -  c_N \, t \, .
%	\end{array}
%	\right. 
%\end{equation}

%	We refer to $p_{-2}=0,\ldots,p_N=0$ as \emph{twisted Ptolemy equations} of $\Tcal_\varphi$.

%	\begin{lemma} \label{lem.lift}
%		The geometirc represnetation $\rho$ admits a lift $\widetilde{\rho} : \pi_1(M_\varphi) \rightarrow \mathrm{SL}_2(\Cbb)$ such $\tr \widetilde{\rho}(\mu)=2$ and $\tr \widetilde{\rho}(\lambda)=-2$
%	\end{lemma}
%	\begin{proof} The fundamental group $\pi_1(M_\varphi)$ has a presentation 
%		\[\pi_1(M_\varphi) = \langle g_1,g_2, \mu \, | \, \mu g^{-1} \mu  =\varphi_\ast (g_1), \mu g^{-1} \mu  =\varphi_\ast (g_1) \rangle\]
%	where $g_1$ and $g_2$ are the generators of $\pi_1(F)$ and $\varphi_\ast : \pi_1(F)\rightarrow \pi_1(F)$ is the isomorphism induced from the monodromy $\varphi$.
%	It is known that the geometric representation $\rho$ has a non-trivial obstruction class \cite{Cal06}, i.e., any lift $\widetilde{\rho}$ satisfies either $\tr \widetilde{\rho}(\mu)=-2$  or $\tr \widetilde{\rho}(\lambda)=-2$. 	
%	\end{proof}

%	It is proved in \cite[Sec.2.3]{Yoon19} that if the obstruction cocycle $\sigma$ matches, then there exists a $\sigma$-Ptolemy assignment $c \in P_\sigma(M_\varphi)$ corresponding to the geometric solution $z_1,\ldots, z_N$ (in terms of Equation~\ref{eqn.shape}).
%	
	\begin{lemma} \label{lem.red}
	 We have 
\[ \det \left( \dfrac{\partial (p^\sigma_{N+1},\, p^\sigma_{N+2},\, 
	p^\sigma_{N+3},\,\overline{p}^\sigma_1,\ldots,\overline{p}^\sigma_N)}{\partial (c_{N+1},\, c_{N+2},\, c_{N+3}, \, c_1,\ldots,c_N)} \right) =\det \left( \dfrac{\partial (p^\sigma_{N+1},\, p^\sigma_{N+2},\, 
	p^\sigma_{N+3},\,\overline{p}^\sigma_1,\ldots,\overline{p}^\sigma_{N-1})}{\partial (c_{N+1},\, c_{N+2},\, c_{N+3}, \, c_1,\ldots,c_{N-1})} \right)\]		
	where the right-hand side is obtained by eliminating the variable $c_N$  by the equation $\overline{p}^\sigma_N =0$.
	\end{lemma}
	\begin{proof}
%	Recall that the equation $\overline{p}^\sigma_N=0$ is given by
%	$\overline{p}^\sigma_N= 	c_{N}-\frac{c_\beta^2+c_\gamma^2}{c_\alpha}=0$   $(\Leftrightarrow c_N= \frac{c_\beta^2+c_\gamma^2}{c_\alpha})$
%	for some indices $\alpha,\beta, \gamma \neq N$. 
	From  the determinant formula for a block matrix with the fact $\partial \overline{p}^\sigma_N / \partial c_N =1$ we have
		\begin{align}
		&\det \left( \dfrac{\partial (p^\sigma_{N+1},\, p^\sigma_{N+2},\, 
			p^\sigma_{N+3},\,\overline{p}^\sigma_1,\ldots,\overline{p}^\sigma_N)}{\partial (c_{N+1},\, c_{N+2},\, c_{N+3}, \, c_1,\ldots,c_N)} \right)\\
		& =\det \left ( \dfrac{\partial (p^\sigma_{N+1},\, p^\sigma_{N+2},\, 
			p^\sigma_{N+3},\,\overline{p}^\sigma_1,\ldots,\overline{p}^\sigma_{N-1})}{\partial (c_{N+1},\, c_{N+2},\, c_{N+3}, \, c_1,\ldots,c_{N-1})}  - 
		\begin{pmatrix}
			\frac{\partial p^\sigma_{N+1}}{c_{N}} \\
			\vdots \\
			\frac{\partial \overline{p}^\sigma_{N-1}}{c_{N}} 
		\end{pmatrix} 
		\begin{pmatrix}
		\frac{\partial \overline{p}^\sigma_N}{c_{N+1}} & \cdots &
		\frac{\partial \overline{p}^\sigma_N}{c_{N-1}} 
		\end{pmatrix} \right ). \label{eqn.induction}
		\end{align}
	Since $\partial \overline{p}^\sigma_k / \partial c_N =0$ for $1 \leq k \leq N-1$,  the matrix in Equation~\eqref{eqn.induction} equals to
	\begin{equation} 
 \dfrac{\partial (p^\sigma_{N+1},\, p^\sigma_{N+2},\, 
	p^\sigma_{N+3},\,\overline{p}^\sigma_1,\ldots,\overline{p}^\sigma_{N-1})}{\partial (c_{N+1},\, c_{N+2},\, c_{N+3}, \, c_1,\ldots,c_{N-1})}-		\begin{pmatrix}
			\frac{\partial p^\sigma_{N+1}}{\partial c_{N}} \frac{\partial \overline{p}^\sigma_N}{\partial c_{N+1}} & \cdots & \frac{\partial p^\sigma_{N+1}}{\partial c_{N}}		\frac{\partial \overline{p}^\sigma_N}{\partial c_{N-1}} 	\\[5pt]
			\frac{\partial p^\sigma_{N+2}}{\partial c_{N}} \frac{\partial \overline{p}^\sigma_N}{\partial c_{N+1}} & \cdots & \frac{\partial p^\sigma_{N+2}}{\partial c_{N}}		\frac{\partial \overline{p}^\sigma_N}{\partial c_{N-1}} 	\\[5pt]	
			\frac{\partial p^\sigma_{N+3}}{\partial c_{N}} \frac{\partial \overline{p}^\sigma_N}{\partial c_{N+1}} & \cdots & \frac{\partial p^\sigma_{N+3}}{\partial c_{N}}		\frac{\partial \overline{p}^\sigma_N}{\partial c_{N-1}} 	\\[5pt]
%			\hline 
			& \bigzero &
		\end{pmatrix}.
	\end{equation}
	Moreover, since  $\frac{\partial \overline{p}^\sigma_N}{\partial c_k}= - \frac{\partial c_N}{\partial c_k}$ for $k \neq N$,
	the above matrix equals to 
	\begin{equation} \label{eqn.induction2}
	  \dfrac{\partial (p^\sigma_{N+1},\, p^\sigma_{N+2},\, 
		p^\sigma_{N+3},\,\overline{p}^\sigma_1,\ldots,\overline{p}^\sigma_{N-1})}{\partial (c_{N+1},\, c_{N+2},\, c_{N+3}, \, c_1,\ldots,c_{N-1})} +		\begin{pmatrix}
		\frac{\partial p^\sigma_{N+1}}{\partial c_{N}} \frac{\partial c_N}{\partial c_{N+1}} & \cdots & \frac{\partial p^\sigma_{N+1}}{\partial c_{N}}		\frac{\partial c_N}{\partial c_{N-1}} 	\\[5pt]
		\frac{\partial p^\sigma_{N+2}}{\partial c_{N}} \frac{\partial c_N}{\partial c_{N+1}} & \cdots & \frac{\partial p^\sigma_{N+2}}{\partial c_{N}}		\frac{\partial c_N}{\partial c_{N-1}} 	\\[5pt]	
		\frac{\partial p^\sigma_{N+3}}{\partial c_{N}} \frac{\partial c_N}{\partial c_{N+1}} & \cdots & \frac{\partial p^\sigma_{N+3}}{\partial c_{N}}		\frac{\partial c_N}{\partial c_{N-1}} 	\\[5pt]
		%			\hline 
		& \bigzero &
	\end{pmatrix}.
	\end{equation}		
	Due to the chain rule, if  we eliminate the variable $c_N$ by using the equation $\overline{p}^\sigma_N =0$, then the matrix~\eqref{eqn.induction2} is equal to  $\frac{\partial (p^\sigma_{N+1}, p^\sigma_{N+2},
		p^\sigma_{N+3},\overline{p}^\sigma_1,\ldots,\overline{p}^\sigma_{N-1})}{\partial (c_{N+1}, c_{N+2}, c_{N+3}, \, c_1,\ldots,c_{N-1})} $
	where $p^\sigma_{N+1}, p^\sigma_{N+2},
	p^\sigma_{N+3}, \overline{p}^\sigma_1,\ldots,\overline{p}^\sigma_{N-1}$ are viewed as functions in $N+2$ variables $c_{N+1}, c_{N+2}, c_{N+3},  c_1,\ldots,c_{N-1}$ (while those in the matrix~\eqref{eqn.induction2} are viewed as functions in $N+3$ variables $c_{N+1}, c_{N+2}, c_{N+3},  c_1,\ldots,c_{N}$).
	\end{proof}
	Continuing the variable reduction of Lemma~\ref{lem.red}, we obtain
	\begin{align}  \allowdisplaybreaks
\tau^\CS(\Tcal_\varphi,t) &= \pm \det \left( \dfrac{\partial (p^\sigma_{N+1},\, p^\sigma_{N+2},\, 
	p^\sigma_{N+3},\,\overline{p}^\sigma_1,\ldots,\overline{p}^\sigma_N)}{\partial (c_{N+1},\, c_{N+2},\, c_{N+3}, \, c_1,\ldots,c_N)} \right) \\
&= \pm \det \left( \dfrac{\partial (p^\sigma_{N+1},\, p^\sigma_{N+2},\, 
	p^\sigma_{N+3})}{\partial (c_{N+1},\, c_{N+2},\, c_{N+3})} \right)\\
&=\pm \det \left ( I-\frac{\partial (c_{N-1}, c_{N-t_n}, c_{N})}{\partial (c_{N+1},c_{N+2}, c_{N+3})} t\right)
	\end{align}
	where $c_{N-t_n}, c_{N-1},  c_N$ are viewed as  functions in $c_{N+1}, c_{N+2}, c_{N+3}$.
	On the other hand, the adjoint twisted Alexander polynomial $\tau(M_\varphi,t)$  is palindromic, namely, $\tau(M_\varphi,t)=\tau(M_\varphi,t^{-1})$ up to multiplication by an element of $\pm t^\Zbb$. See e.g. \cite{Kitano96,KL99}.
	Therefore, combining the above with Theorem~\ref{thm.ATAP}, we obtain
	\begin{equation}
	\tau^\CS(\Tcal_\varphi,t)
	=\det \left ( I-\frac{\partial (c_{N-1}, c_{N-t_n}, c_{N})}{\partial (c_{N+1},c_{N+2}, c_{N+3})} t\right)=\tau(M_\varphi,t^{-1})=\tau(M_\varphi,t)
	\end{equation}
	up to multiplication by an element of $\pm t^\Zbb$. This proves Theorem~\ref{thm.main2}.
	
	\subsection{Proof of Corollary~\ref{thm.main3}} \label{sec.proof2}
	
		Combining  Theorem~\ref{thm.main2} with \cite[Thm.1.7]{GY21}, we obtain  	\[\tau^\CS_\lambda(\Tcal_\varphi) = \tau_\lambda(M_\varphi)\] where $\lambda$ is a peripheral curve $\lambda$ satisfying $\alpha(\lambda)=0$.
		Here
		 $\tau^\CS_\lambda(\Tcal_\varphi$) and $ \tau_\lambda(M_\varphi)$ are the 1-loop invariant of $\Tcal_\varphi$ and the adjoint Reidemeister torsion of $M_\varphi$, respectively, with respect to $\lambda$.
	On the other hand, in \cite[Sec.5.3]{Siejakowski2017} it is proved that if Conjecture~\ref{conj.1loop} holds for one peripheral curve, then it does for all peripheral curves, hence we obtain Corollary~\ref{thm.main3}.

	\bibliographystyle{alpha}
	\bibliography{main}

\begin{thebibliography}{BDHP19}

\bibitem[BDHP19]{BDHP19}
Leo Benard, Jerome Dubois, Michael Heusener, and Joan Porti.
\newblock {Asymptotics of twisted Alexander polynomials and hyperbolic volume}.
\newblock {\em preprint arXiv:1912.12946}, 2019.

\bibitem[Cal06]{Cal06}
Danny Calegari.
\newblock Real places and torus bundles.
\newblock {\em Geom. Dedicata}, 118:209--227, 2006.

\bibitem[CS83]{CS83}
Marc Culler and Peter~B. Shalen.
\newblock Varieties of group representations and splittings of {$3$}-manifolds.
\newblock {\em Ann. of Math. (2)}, 117(1):109--146, 1983.

\bibitem[DG13]{DG1}
Tudor Dimofte and Stavros Garoufalidis.
\newblock The quantum content of the gluing equations.
\newblock {\em Geom. Topol.}, 17(3):1253--1315, 2013.

\bibitem[FG06]{fock2006moduli}
Vladimir Fock and Alexander Goncharov.
\newblock Moduli spaces of local systems and higher {T}eichm\"{u}ller theory.
\newblock {\em Publ. Math. Inst. Hautes \'{E}tudes Sci.}, (103):1--211, 2006.

\bibitem[FH82]{FH82}
W.~Floyd and A.~Hatcher.
\newblock Incompressible surfaces in punctured-torus bundles.
\newblock {\em Topology Appl.}, 13(3):263--282, 1982.

\bibitem[Fri88]{Fried86}
David Fried.
\newblock Counting circles.
\newblock In {\em Dynamical systems ({C}ollege {P}ark, {MD}, 1986--87)}, volume
  1342 of {\em Lecture Notes in Math.}, pages 196--215. Springer, Berlin, 1988.

\bibitem[GTZ15]{GTZ}
Stavros Garoufalidis, Dylan~P. Thurston, and Christian~K. Zickert.
\newblock The complex volume of {${\rm SL}(n,\Bbb{C})$}-representations of
  3-manifolds.
\newblock {\em Duke Math. J.}, 164(11):2099--2160, 2015.

\bibitem[Gue06]{Gue06}
Fran\c{c}ois Gueritaud.
\newblock On canonical triangulations of once-punctured torus bundles and
  two-bridge link complements.
\newblock {\em Geom. Topol.}, 10:1239--1284, 2006.
\newblock With an appendix by David Futer.

\bibitem[Guk05]{Guk05}
Sergei Gukov.
\newblock Three-dimensional quantum gravity, {C}hern-{S}imons theory, and the
  {A}-polynomial.
\newblock {\em Comm. Math. Phys.}, 255(3):577--627, 2005.

\bibitem[GY21]{GY21}
Stavros Garoufalidis and Seokbeom Yoon.
\newblock {Twisted Neumann--Zagier matrices}.
\newblock {\em preprint arXiv:2109.00379}, 2021.

\bibitem[Kit96]{Kitano96}
Teruaki Kitano.
\newblock Twisted {A}lexander polynomial and {R}eidemeister torsion.
\newblock {\em Pacific J. Math.}, 174(2):431--442, 1996.

\bibitem[KL99]{KL99}
Paul Kirk and Charles Livingston.
\newblock Twisted {A}lexander invariants, {R}eidemeister torsion, and
  {C}asson-{G}ordon invariants.
\newblock {\em Topology}, 38(3):635--661, 1999.

\bibitem[Lac03]{lackenby2003canonical}
Marc Lackenby.
\newblock The canonical decomposition of once-punctured torus bundles.
\newblock {\em Comment. Math. Helv.}, 78(2):363--384, 2003.

\bibitem[Neu92]{Neumann04}
Walter Neumann.
\newblock Combinatorics of triangulations and the {C}hern-{S}imons invariant
  for hyperbolic {$3$}-manifolds.
\newblock In {\em Topology '90 ({C}olumbus, {OH}, 1990)}, volume~1 of {\em Ohio
  State Univ. Math. Res. Inst. Publ.}, pages 243--271. de Gruyter, Berlin,
  1992.

\bibitem[Neu04]{neumann2004extended}
Walter Neumann.
\newblock Extended {B}loch group and the {C}heeger-{C}hern-{S}imons class.
\newblock {\em Geom. Topol.}, 8:413--474, 2004.

\bibitem[NZ85]{NZ}
Walter Neumann and Don Zagier.
\newblock Volumes of hyperbolic three-manifolds.
\newblock {\em Topology}, 24(3):307--332, 1985.

\bibitem[Pen12]{penner2012decorated}
Robert~C. Penner.
\newblock {\em Decorated {T}eichm\"{u}ller theory}.
\newblock QGM Master Class Series. European Mathematical Society (EMS),
  Z\"{u}rich, 2012.

\bibitem[Por97]{Porti:torsion}
Joan Porti.
\newblock Torsion de {R}eidemeister pour les vari\'{e}t\'{e}s hyperboliques.
\newblock {\em Mem. Amer. Math. Soc.}, 128(612):x+139, 1997.

\bibitem[Sie17]{Siejakowski2017}
Rafa{\l} Siejakowski.
\newblock {Infinitesimal gluing equations and the adjoint hyperbolic
  Reidemeister torsion}.
\newblock {\em preprint arXiv:1710.02109}, 2017.

\bibitem[Sik12]{Sik12}
Adam~S. Sikora.
\newblock Character varieties.
\newblock {\em Trans. Amer. Math. Soc.}, 364(10):5173--5208, 2012.

\bibitem[Thu77]{Thurston}
William Thurston.
\newblock {\em The geometry and topology of 3-manifolds}.
\newblock Universitext. Springer-Verlag, Berlin, 1977.
\newblock \url{http://msri.org/publications/books/gt3m}.

\bibitem[Wei64]{Wei64}
Andr\'{e} Weil.
\newblock Remarks on the cohomology of groups.
\newblock {\em Ann. of Math. (2)}, 80:149--157, 1964.

\bibitem[Yoo19]{Yoon19}
Seokbeom Yoon.
\newblock The volume and {C}hern-{S}imons invariant of a {D}ehn-filled
  manifold.
\newblock {\em Topology Appl.}, 256:208--227, 2019.

\bibitem[Zic09]{Zic09}
Christian~K. Zickert.
\newblock The volume and {C}hern-{S}imons invariant of a representation.
\newblock {\em Duke Math. J.}, 150(3):489--532, 2009.

\end{thebibliography}

	\end{document}